\documentclass[12pt,reqno]{article}

\usepackage[usenames]{color}
\usepackage{amssymb}
\usepackage{graphicx}
\usepackage{amscd}

\usepackage[colorlinks=true,
linkcolor=webgreen,
filecolor=webbrown,
citecolor=webgreen]{hyperref}

\definecolor{webgreen}{rgb}{0,.5,0}
\definecolor{webbrown}{rgb}{.6,0,0}

\usepackage{color}
\usepackage{float}

\usepackage{graphics,amsmath,amssymb}
\usepackage{amsthm}
\usepackage{amsfonts}
\usepackage{latexsym}

\newcommand{\seqnum}[1]{\href{http://oeis.org/#1}{\underline{#1}}}

\begin{document}

\theoremstyle{plain}
\newtheorem{theorem}{Theorem}
\newtheorem{corollary}[theorem]{Corollary}
\newtheorem{lemma}[theorem]{Lemma}
\newtheorem{proposition}[theorem]{Proposition}
\newtheorem{obs}[theorem]{Observation}

\theoremstyle{definition}
\newtheorem{definition}[theorem]{Definition}
\newtheorem{example}[theorem]{Example}
\newtheorem{conjecture}[theorem]{Conjecture}

\theoremstyle{remark}
\newtheorem{remark}[theorem]{Remark}

\begin{center}
\vskip 1cm
{\LARGE\bf Reinterpreting the Middle-Levels Theorem via Natural Enumeration of Ordered Trees}

\vskip 1cm
\large
Italo J. Dejter\\
University of Puerto Rico\\
Rio Piedras, PR 00936-8377\\
\href{mailto:italo.dejter@gmail.com}{\tt italo.dejter@gmail.com} \\
\end{center}
 
\begin{abstract}\noindent 
Let $0<k\in\mathbb{Z}$. A reinterpretation of the proof of existence of Hamilton cycles in the middle-levels graph $M_k$ induced by the vertices of the $(2k+1)$-cube representing the $k$- and $(k+1)$-subsets of $\{0,\ldots,2k\}$ is given via an associated dihedral quotient graph of $M_k$ whose vertices represent the ordered (rooted) trees of order $k+1$ and size $k$.
\end{abstract}

\section{Introduction}\label{s1}

Let $0<k\in\mathbb{Z}$. 
The {\it middle-levels graph} $M_k$  \cite{KT} is the subgraph induced by the $k$-th and $(k+1)$-th levels (formed by the $k$- and $(k+1)$-subsets of $[2k+1]=\{0,\ldots,2k\}$) in the Hasse diagram \cite{Savage} of the Boolean lattice $2^{[2k+1]}$.
The dihedral group $D_{4k+2}$ acts on $M_k$ via translations mod $2k+1$ (Section~\ref{s4}) and complemented reversals (Section~\ref{s6}). 
The sequence $\mathcal S$ \cite{oeis} \seqnum{A239903} of {\it restricted-growth strings} or {\it RGS's} 
(\cite{Arndt} page 325, \cite{Stanley} page 224 item (u)) will be shown to unify the presentation of all $M_k$'s. 
In fact, the first $C_k$ terms of $\mathcal S$ will stand for the orbits of $V(M_k)$ under natural $D_{4k+2}$-action, where $C_k=\frac{(2k)!}{k!(k+1)!}$ is the $k$-th Catalan number \cite{oeis} \seqnum{A000108}. This will provide a reinterpretation (Section~\ref{ultimo}) of the Middle-Levels Theorem on the existence of Hamilton cycles in $M_k$ \cite{M,gmn} via $k${\it -germs} (Section~\ref{srgg}) of RGS's. For the history of this theorem, \cite{M} may be consulted, but the conjecture answered by  it was learned from B. Bollob\'as on January 23, 1983, who mentioned it as a P. Erd\"os conjecture.

 In Section~\ref{s7}, the cited $D_{4k+2}$-action on $M_k=(V(M_k),E(M_k))$ allows to project $M_k$ onto a quotient graph $R_k$ whose vertices stand for the first $C_k$ terms of $\mathcal S$ via the lexical-matching colors \cite{KT} (or {\it lexical colors}) $0,1,\ldots,k$ on the $k+1$ edges incident to each vertex (Section~\ref{s10}).
In preparation, RGS's $\alpha$ are converted in Section~\ref{s3} into $(2k+1)$-strings $F(\alpha)$, composed by the $k+1$ lexical colors and $k$ asterisks, representing all ordered (rooted) $k$-edge trees
(Remark~\ref{re}) via  a ``Castling'' procedure that facilitates enumeration. These trees (encoded as $F(\alpha)$) are shown to represent the vertices of $R_k$ via ``Uncastling'' procedure (Section~\ref{s9}).

In Section~\ref{2fact}, the 2-factor $W_{01}^k$ of $R_k$ determined by the colors 0 and 1 is analyzed from the RGS-dihedral action viewpoint. From this, $W_{01}^k$ is seen in Section~\ref{ultimo} to morph into Hamilton cycles of $M_k$ via symmetric differences with 6-cycles whose presentation is alternate to that of \cite{gmn}.
In particular, an integer sequence ${\mathcal S}_0$ is shown to exist such that, for each integer $k>0$,
the neighbors via color $k$ of the RGS's in $R_k$ ordered as in $\mathcal S$ correspond to an idempotent permutation on the first $C_k$ terms of ${\mathcal S}_0$. This and related properties hold for colors $0,1,\ldots,k$ (Theorem~\ref{nt}
 and Remarks~\ref{pro1}-\ref{p11}) in part reflecting and extending from Observation~\ref{lem} in Section~\ref{2fact} properties of 
plane trees (i.e., classes of ordered trees under root rotation). 
Moreover, Section~\ref{rootk} considers suggestive symmetry properties by reversing the designation of the roots in the ordered trees so that each lexical-color $i$ adjacency ($0\le i\le k$) can be seen from the lexical-color $(k-i)$ viewpoint.   

\section{Restricted-Growth Strings and $k$-Germs}\label{srgg}

Let $0<k\in\mathbb{Z}$. The sequence of (pairwise different) RGS's ${\mathcal S}=(\beta(0),\ldots,\beta(17),\ldots)=$
$$(0,1,10,11,12,100,101,110,111,112,120,121,122,123,1000,1001,1010,1011,\ldots)$$
has the lengths of its contiguous pairs $(\beta(i-1),\beta(i))$ constant unless $i=C_k$ for $1<k\in\mathbb{Z}$, in which case 
 $\beta(i-1)=\beta(C_k-1)=12\cdots k$ and $\beta(i)=\beta(C_k)=10^k=10\cdots 0$.

To view the continuation of the sequence $\mathcal S$, each RGS $\beta=\beta(m)$ will be transformed, for every $k\in\mathbb{Z}$ such that $k\ge$ length$(\beta)$, into a $(k-1)$-string $\alpha=a_{k-1}a_{k-2}\cdots a_2a_1$
 by prefixing $k-$ length$(\beta)$ zeros to $\beta$. We say such $\alpha$ is a $k$-{\it germ}. More generally, 
 a $k$-{\it germ} $\alpha$ ($1<k\in\mathbb{Z}$) is formally defined to be a $(k-1)$-string $\alpha=a_{k-1}a_{k-2}\cdots a_2a_1$ such that:
\begin{enumerate}\item[\bf(1)] the leftmost position (called position $k-1$) of $\alpha$ contains the entry $a_{k-1}\in\{0,1\}$;
\item[\bf(2)] given $1<i<k$,
the entry $a_{i-1}$ (at position $i-1$) satisfies $0\le a_{i-1}\le a_i+1$.
\end{enumerate}
Every $k$-germ $a_{k-1}a_{k-2}\cdots a_2a_1$ yields the $(k+1)$-germ $0a_{k-1}a_{k-2}\cdots$ $a_2a_1$.
A {\it non-null RGS} is obtained by stripping a $k$-germ $\alpha=a_{k-1}a_{k-2}\cdots a_1$ $\ne 00\cdots 0$ off all the null entries to the left of its leftmost position containing a 1. Such a non-null RGS is again denoted $\alpha$. We also say that the {\it null RGS} $\alpha=0$ corresponds to every null $k$-germ $\alpha$, for $0<k\in\mathbb{Z}$. (We use the same notations $\alpha=\alpha(m)$ and $\beta=\beta(m)$ to denote both a $k$-germ and its associated RGS).

The $k$-germs are ordered as follows. Given 2 $k$-germs, say
$\alpha=a_{k-1}\cdots a_2a_1$ and $\beta=b_{k-1}\cdots b_2b_1,$ where $\alpha\ne \beta$, we say that $\alpha$ precedes $\beta$, written $\alpha<\beta$, whenever either:
\begin{enumerate}
\item[\bf (i)] $a_{k-1} < b_{k-1}$ or
\item[\bf (ii)] $\exists i\in(0,k)\cap\mathbb{Z}$ such that $a_i < b_i$, with $a_j=b_j$, $\forall j\in(i,k)\cap\mathbb{Z}$.
\end{enumerate}
The resulting order on $k$-germs $\alpha(m)$, ($m\le C_k$), corresponding biunivocally (via the assignment $m\rightarrow\alpha(m)$) with the natural order on $m$, yields a listing that we call the {\it natural ($k$-germ)  enumeration}.
Note that there are exactly $C_k$ $k$-germs $\alpha=\alpha(m)<10^k$, $\forall k>0$. 

\section{Castling of Ordered Trees}\label{s3}

\begin{theorem}\label{thm1} To each $k$-germ $\alpha=a_{k-1}\cdots a_1\ne 0^{k-1}$ with rightmost entry $a_i\ne 0$ ($k>i\ge 1$)
corresponds a $k$-germ $\beta(\alpha)=b_{k-1}\cdots b_1\!<\alpha$ with $b_i= a_i-1$ and $a_j=b_j$, ($j\ne i$).
All $k$-germs form an ordered tree ${\mathcal T}_k$ rooted at $0^{k-1}$, each $k$-germ $\alpha\ne0^{k-1}$ with $\beta(\alpha)$ as parent.
\end{theorem}

\begin{proof} The statement, illustrated for $k=2,3,4$ by means of the first 3 columns of Table I, is straightforward. Table I also serves as illustration to the proof of Theorem~\ref{thm2}, below.
\end{proof}

By representing ${\mathcal T}_k$ with each node $\beta$ having its children $\alpha$ enclosed between parentheses following $\beta$ and separating siblings with commas, we can write: $${\mathcal T}_4=000(001,010(011(012)),100(101,110(111(121)),120(121(122(123))))).$$

\begin{center}TABLE I
$$\begin{array}{||r||c|c|c|c|c|c|c|c||}
m&\;\;\alpha &\;\;\beta & F(\beta) & i & W^i\,|\,X\,|\,Y\,|\,Z^i & W^i\,|\,Y\,|\,X\,|\,Z^i & F(\alpha) & \alpha\\\hline\hline
0&\;\;0 &\;\; -  &    -            & - &    -        &   -         & 012\!*\!* & \;\;0 \\\hline
1&\;\;1 &\;\;0  & 012\!*\!*       & 1 & 0\,|\,1\,|\,2\!*|*     & 0\,|\,2\!*|\,1\,|* & 02\!*\!1* & \;\;1 \\\hline\hline
0&00 & -   &    -            & - &    -        &   -                           & 012\,3**\,* & 00 \\\hline
1&01 & 00 & 0123***     & 1 & 0|1|23**|* & 0|23**|1|* & 023**\,1\,* & 01 \\\hline
2&10 & 00 & 0123***     & 2 & 01|2|\,3*|** & 01|3*|2|** & 013*2** & 10 \\\hline
3&11 & 10 & 013*2** & 1 & 0|13*|2*|* & 0|2*|13*|* & 02*13** & 11 \\\hline
4&12 & 11 & 02*13** & 1 & 0|2*1|3*|* & 0|3*|2*3|* & 03*2*1* & 12 \\\hline\hline
0&000 & -   & -         & - &            - &            - & 01234**** & 000\\\hline
1&001 & 000 & 01234**** & 1 & 0|1|234***|\,* & 0|234***|1|* & 0234***1\,* & 001\\\hline
2&010 & 000 & 01234**** & 2 & 01|2|34**|*\,* & 01|34**|2|*\,* & 0134**\,2** & 010\\\hline
3&011 & 010 & 0134**2** & 1 & 0|134**|2*|\,* & 0|2*|134**|\,* & 02*134**\,* & 011\\\hline
4&012 & 011 & 02*134*** & 1 & 0|2*1|34**|\,* & 0|34**|2*1|\,* & 034**2*1\,* & 012\\\hline
5&100 & 000 & 01234**** & 3 & 012|3|4*|**\,* & 012|4*|3|**\,* & 0124*3**\,* & 100\\\hline
6&101 & 100 & 0124*3*** & 1 & 0|1|24*3**|\,* & 0|24*3**|1\,|\,* & 024*3**\,1\,* & 101\\\hline
7&110 & 100 & 0124*3*** & 2 & 01|24*|3*|**   & 01|3*|24*|** & 013*24**\,* & 110\\\hline
8&111 & 110 & 013*24*** & 1 & 0|13*|24**|*   & 0|24**|\,13*|\,* & 024**\,13** & 111\\\hline
9&112 & 111 & 024**13** & 1 & 0|24**1|3*|*   & 0|3*|24**\,1|\,* & 03*24**\,1\,* & 112\\\hline
10&120 & 110 & 013*24*** & 2 & 01|3*2|4*|** & 01|4*|3*2|** & 014*3*2** & 120\\\hline
11&121 & 120 & 014*3*2** &  1 & 0|14*3*|2*|* & 0|2*|14*3*|* & 02*14*3** & 121\\\hline
12&122 & 121 & 02*14*3** & 1 & 0|2*34*|3*|* & 0|3*|2*14*|* & 03*2*14** & 122\\\hline
13&123 & 122 & 03*2*14** & 1 & 0|3*2*1|4*|* & 0|4*|3*2*1|* & 04*3*2*1* & 123\\\hline\hline
\end{array}$$\end{center}

\begin{theorem}\label{thm2}
To each $k$-germ $\alpha=a_{k-1}\cdots a_1$ corresponds a $(2k+1)$-string $F(\alpha)=f_0f_1\cdots f_{2k}$ whose entries are both the numbers $0,1,\ldots,k$ (once each) and $k$ asterisks ($*$) and such that:

\vspace*{1mm}

{\bf(A)}
$F(0^{k-1})=012\cdots(k-2)(k-1)k*\cdots *$;

\vspace*{1mm}

{\bf(B)}
if $\alpha\ne 0^{k-1}$, then $F(\alpha)$ is obtained from $F(\beta)=F(\beta(\alpha))=h_0h_1\cdots h_{2k}$ by means of

\hspace*{8mm}the following ``Castling Procedure'' steps:

\vspace*{1mm}

\noindent{\bf 1.}
let $W^i=h_0h_1\cdots h_{i-1}=f_0f_1\cdots f_{i-1}$ and $Z^i=h_{2k-i+1}\cdots h_{2k-1}h_{2k}=f_{2k-i+1}\cdots f_{2k-1}f_{2k}$

be respectively the initial and terminal substrings of length $i$ in $F(\beta)$;

\vspace*{1mm}

\noindent{\bf 2.} let $\Omega>0$ be the leftmost entry of the substring $U=F(\beta)\setminus(W^i\cup Z^i)$ and consider the

concatenation $U=X|Y$, with $Y$
starting at entry $\Omega+1$; then, $F(\beta)=W^i|X|Y|Z^i$;

\vspace*{1mm}

\noindent{\bf 3.} set $F(\alpha)=W^i|Y|X|Z^i$.

\vspace*{1mm}

\noindent In particular:
\begin{enumerate}
\item[\bf(a)] the leftmost entry of each $F(\alpha)$ is $0$; $k*$ is a substring of $F(\alpha)$, but $*k$ is not;
\item[\bf(b)] a number to the immediate right of any $b\in[0,k)$ in $F(\alpha)$ is larger than $b$;
\item[\bf(c)] $W^i$ is an (ascending) number $i$-substring and $Z^i$ is formed by $i$ of the $k$ asterisks.
\end{enumerate}\end{theorem}

\begin{proof} Let $\alpha=a_{k-1}\cdots a_1\ne 0^{k-1}$ be a $k$-germ. In the sequence of applications of 1-3 along the path from root $0^{k-1}$ to $\alpha$ in ${\mathcal T}_k$, unit augmentation of $a_i$ for larger values of $i$, ($0<i<k$), must occur earlier, and then in strictly descending order of the entries $i$ of the intermediate $k$-germs. As a result, the length of the inner substring $X|Y$ is kept non-decreasing after each application. This is illustrated in Table I below, where the order of presentation of $X$ and $Y$ is reversed in successively decreasing steps. In the process, (a)-(c) are seen to be fulfilled.

The 3 successive subtables in Table I have $C_k$ rows each, where $C_2=2$, $C_3=5$ and $C_4=14$; in the subtables, the $k$-germs $\alpha$ are shown both on the second and last columns via natural enumeration in the first column; the images $F(\alpha)$ of those $\alpha$ are on the penultimate column;
 the remaining columns in the table are filled, from the second row on, as follows:
{\bf (i)} $\beta=\beta(\alpha)$, arising in Theorem~\ref{thm1};
{\bf (ii)} $F(\beta)$, taken from the penultimate column in the previous row;
{\bf (iii)} the length $i$ of $W^i$ and $Z^i$ ($k-1\ge i\ge 1$);
{\bf (iv)} the decomposition $W^i|Y|X|Z^i$ of $F(\beta)$;
{\bf (v)} the decomposition $W^i|X|Y|Z^i$ of $F(\alpha)$, re-concatenated in the following, penultimate, column as $F(\alpha)$, with $\alpha=F^{-1}(F(\alpha))$ in the last column.  
\end{proof}

\begin{remark}\label{re} 
As in the case of ${\mathcal T}_k$ in Theorem~\ref{thm1}, an {\it ordered (rooted) tree} \cite{gmn} is a tree $T$ with: {\bf(a)} a specified node $v_0$ as the root of $T$; {\bf(b)} an embedding of $T$ into the plane with $v_0$ on top; {\bf(c)} the edges between the $j$- and $(j+1)$-levels of $T$ (formed by the nodes at distance $j$ and $(j+1)$ from $v_0$, where $0\le j<$ height$(T)$) having (parent) nodes at the $j$-level above their children at the $(j+1)$-level;  {\bf(d)} the children in (c) ordered in a left-to-right fashion. 
Each $k$-edge ordered tree $T$ is both represented by a $k$-germ $\alpha$ and by its associated $(2k+1)$-string $F(\alpha)$, so we write $T=T_\alpha$.
 In fact, we perform a depth first search (DFS, or $\rightarrow$DFS) on $T$ with its vertices from $v_0$ downward denoted as $v_i$ ($i=0,1,\ldots,k$) in a right-to-left breadth-first search ($\leftarrow$BFS) way.
Such DFS yields the claimed $F(\alpha)$ by writing successively:
\begin{enumerate}
\item[{\bf(i)}] the subindex $i$ of each $v_i$ in the $\rightarrow$DFS downward appearance and 

\item[{\bf(ii)}] an asterisk for the edge $e_i$ ending at each child $v_i$ in the $\rightarrow$DFS upward appearance. 
\end{enumerate}
\noindent 
Then, we write: $F(T_\alpha)=F(\alpha).$  
 Now, Theorem~\ref{thm1} can be taken as a {\it tree-surgery transformation} from $T_\beta$ onto $T_\alpha$ for each $k$-germ $\alpha\ne 0^{k-1}$ via the vertices $v_i$ and edges $e_i$ (whose parent vertices are generally reattached in different ways). This remark is used in Sections~\ref{2fact}-\ref{rootk} in reinterpreting the Middle-Levels  Theorem.
(In Section~\ref{rootk}, an alternate viewpoint on ordered trees taking $a_k$ as the root instead of $a_0$ is considered). 
\end{remark}

To each $F(\alpha)$ corresponds a binary $n$-string $\theta(\alpha)$ of weight $k$ obtained by replacing each number in $[k]=\{0,1,\ldots,k-1\}$ by 0 and each asterisk $*$ by 1. By attaching the entries of $F(\alpha)$ as subscripts to the corresponding entries of $\theta(\alpha)$, a subscripted binary $n$-string $\hat{\theta}(\alpha)$ is obtained, as shown for $k=2,3$ in the 4th column of Table II below. Let $\aleph(\theta(\alpha))$ be given by the {\it complemented reversal} of $\theta(\alpha)$, that is:
\begin{eqnarray}\label{d2}\mbox{if }\theta(\alpha)=a_0a_1\cdots a_{2k}\mbox{, then }\aleph(\theta(\alpha ))=\bar{a}_{2k}\cdots\bar{a}_1\bar{a}_0,\end{eqnarray} where $\bar{0}=1$ and $\bar{1}=0$. A subscripted version
$\hat{\aleph}$ of $\aleph$ is obtained for $\hat{\theta}(\alpha)$, as shown in the fifth column of Table II, with the subscripts of $\hat{\aleph}$ reversed with respect to those of $\aleph$.
Each image of a $k$-germ $\alpha$ under $\aleph$ is an $n$-string of weight $k+1$ and has the 1's indexed with numeric subscripts and the 0's indexed with the asterisk subscript.
The number subscripts reappear from Section~\ref{s8} on
 as lexical colors \cite{KT} for the graphs $M_k$.

\begin{center}TABLE II
$$\begin{array}{||c||c|c|c|c|c||}
m&\alpha&\theta(\alpha)&\hat{\theta}(\alpha)&\hat{\aleph}(\theta(\alpha))=\aleph(\hat{\theta}(\alpha))&\aleph(\theta(\alpha))\\
\hline\hline
0&\;0&00011&0_00_10_21_*1_*&0_*0_*1_21_11_0&00111  \\\hline
1&\;1&00101& 0_00_21_*0_11_*&0_*1_10_*1_21_0&01011\\\hline\hline
0&00&0000111&0_00_10_20_31_*1_*1_*&0_*0_*0_*1_31_21_11_0&0001111 \\\hline
1&01&0001101&0_00_20_31_*1_*0_11_*&0_*1_10_*0_*1_31_21_0&0100111 \\\hline
2&10&0001011&0_00_10_31_*0_21_*1_*&0_*0_*1_20_*1_31_11_0&0010111 \\\hline
3&11&0010011&0_00_21_*0_10_31_*1_*&0_*0_*1_31_10_*1_21_0&0011011 \\\hline
4&12&0010101&0_00_31_*0_21_*0_11_*&0_*1_10_*1_20_*1_31_0&0101011 \\\hline\hline
\end{array}$$\end{center}

\section{Translations mod $n=2k+1$}\label{s4}

Let $n=2k+1$. The $n$-{\it cube graph} $H_n$ is the Hasse diagram of the Boolean lattice $2^{[n
]}$ on the set $[n]=\{0,\ldots,n-1\}$. It is convenient to express each vertex $v$ of $H_n$ in 3 different equivalent ways:
\begin{enumerate}
\item[\bf (a)] ordered set
$A=\{a_0,a_1,\ldots,a_{j-1}\}=a_0a_1\cdots a_{j-1}\subseteq [n]$ that $v$
represents, ($0<j\le n$);
\item[\bf (b)] characteristic binary $n$-vector $B_A=(b_0,b_1,\ldots,b_{n-1})$ of ordered set $A$ in (a)
above, where $b_i=1$ if and only if $i\in A$, ($i\in[n]$);
\item[\bf (c)] polynomial $\epsilon_A(x)=b_0+b_1x+\cdots
+b_{n-1}x^{n-1}$ associated to $B_A$ in (b) above.
\end{enumerate}
\noindent Ordered set $A$ and vector $B_A$ in (a) and (b) respectively are written for short as $a_0a_1\cdots a_{j-1}$ and $b_0b_1\cdots b_{n-1}$.
$A$ is said to be the {\it support} of $B_A$. 

For each $j\in[n]$, let
$L_j=\{A\subseteq[n]\mbox{ with }|A|=j\}$ be the $j$-{\it level} of $H_n$.
Then, $M_k$ is the subgraph of $H_n$ induced by $L_k\cup L_{k+1}$, for $1\le k\in\mathbb{Z}$. By viewing the elements of $V(M_k)=L_k\cup L_{k+1}$ as polynomials, as in (c) above, a regular (i.e., free and transitive) {\it translation mod} $n$
action $\Upsilon'$ of $\mathbb{Z}_n$ on $V(M_k)$ is seen to exist, given by:
\begin{eqnarray}\label{d3}\Upsilon':\mathbb{Z}_n\times V(M_k)\rightarrow V(M_k)\mbox{, with   }\Upsilon'(i,v)=v(x)x^i\mbox{ (mod }1+x^n),\end{eqnarray}
where $v\in V(M_k)$ and $i\in\mathbb{Z}_n$.  Now, $\Upsilon'$ yields a quotient graph $M_k/\pi$ of $M_k$, where $\pi$ stands for
the equivalence relation on $V(M_k)$ given by: $$\epsilon_A(x)\pi\epsilon_{A'}(x)\Longleftrightarrow\exists\,i\in\mathbb{Z}\mbox{ with }\epsilon_{A'}(x)\equiv x^i\epsilon_A(x)\mbox{ (mod }1+x^n),$$ with $A,A'\in V(M_k)$.
This is to be used in the proofs of Theorems~\ref{th4} and~\ref{thm23}. Clearly, $M_k/\pi$ is the graph whose vertices are the equivalence classes of $V(M_k)$ under $\pi$.
Also, $\pi$ induces a partition of $E(M_k)$ into equivalence classes, to be taken as the edges of $M_k/\pi$.

\section{Complemented Reversals}\label{s6}

Let $(b_0b_1\cdots b_{n-1})$ denote the class of $b_0b_1\cdots b_{n-1}\in L_i$ in $L_i/\pi$. Let $\rho_i:L_i\rightarrow L_i/\pi$ be the canonical
projection given by assigning $b_0b_1\cdots b_{n-1}$ to $(b_0b_1\cdots b_{n-1})$, for $i\in\{k,k+1\}$.
The definition of $\aleph$ in display~(\ref{d2}) is easily extended to a bijection, again denoted $\aleph$, from $L_k$ onto $L_{k+1}$. Let
$\aleph_\pi:L_k/\pi\rightarrow L_{k+1}/\pi$ be given by
$\aleph_\pi((b_0b_1\cdots b_{n-1}))=(\bar{b}_{n-1}\cdots\bar{b}_1\bar{b_0})$. Observe $\aleph_\pi$ is a bijection. Notice the commutative identities $\rho_{k+1}\aleph=\aleph_\pi\rho_k$ and $\rho_k\aleph^{-1}=\aleph_\pi^{-1}\rho_{k+1}$.

The following geometric representations will be handy. List vertically the vertex parts $L_k$ and $L_{k+1}$ of $M_k$ (resp. $L_k/\pi$ and $L_{k+1}/\pi$ of $M_k/\pi$) so as to display a splitting of $V(M_k)=L_k\cup L_{k+1}$ (resp. $V(M_k)/\pi=L_k/\pi\cup L_{k+1}/\pi$) into pairs, each pair contained in a horizontal line, the 2 composing vertices of such pair equidistant from a vertical line $\phi$ (resp. $\phi/\pi$, depicted through $M_2/\pi$ on the left of Figure 1, Section~\ref{s7} below). In addition, we impose that
each resulting horizontal vertex pair in $M_k$ (resp. $M_k/\pi$) be of the form $(B_A,\aleph(B_A))$ (resp. $((B_A),(\aleph(B_A))=\aleph_\pi((B_A)))$), disposed from left to right at both sides of $\phi$. A non-horizontal edge  of $M_k/\pi$ will be said to be a {\it skew edge}.

\begin{theorem}\label{thm3}
To each skew edge $e=(B_{A})(B_{A'})$ of $M_k/\pi$ corresponds another skew edge $\aleph_\pi((B_{A}))\aleph^{-1}_\pi((B_{A'}))$ obtained from $e$ by reflection on the line $\phi/\pi$. Moreover:

{\bf (i)} the skew edges of $M_k/\pi$ appear in pairs, with the endpoints of the edges in each pair 

\hspace*{5mm} forming 2 horizontal pairs of vertices equidistant from $\phi/\pi$;

{\bf (ii)} each horizontal edge of $M_k/\pi$ has multiplicity equal either to 1 or to 2.
\end{theorem}

\begin{proof}
The skew edges $B_{A}B_{A'}$ and $\aleph^{-1}(B_{A'})\aleph(B_{A})$ of $M_k$ are
reflection of each other about $\phi$. Their endopoints form 2 horizontal pairs $(B_{A},\aleph(B_{A'}))$ and $(\aleph^{-1}(B_{A}),B_{A'})$ of vertices. Now, $\rho_k$ and $\rho_{k+1}$ extend together to a covering graph map $\rho:M_k\rightarrow M_k/\pi$, since the edges accompany the
projections correspondingly, exemplified for $k=2$ as follows:
$$^{\;\;\;\;\,\aleph((B_A))=\;\;\;\;\aleph((00011))=\;\;\;\;\aleph(\{00011,10001,11000,01100,00110\})=\{00111,01110,11100,11001,10011\}=(00111),}
_{\aleph^{-1}((B_A'))=\aleph^{-1}((01011))=\aleph^{-1}(\{01011,10110,10110,11010,10101\})=\{00101,10010,01001,10100,01010\}=(00101).}$$
Here, the order of the elements in the image of class $(00011)$ (resp. $(01011)$) mod $\pi$ under $\aleph$ (resp. $\aleph^{-1}$) are shown reversed, from right to left (cyclically between braces, continuing on the right once one reaches the leftmost brace). Such reversal holds for every $k>2$:
$$^{\;\;\;\;\,\aleph((B_A))=\hspace*{3.4mm}\aleph((b_0\cdots b_{2k}))=\hspace*{3.4mm}\aleph(\{b_0\cdots b_{2k},\;b_{2k}\ldots b_{2k-1},\;\ldots,\;b_1\cdots
b_0\})=
\{\bar{b}_{2k}\cdots\bar{b}_0,\;\bar{b}_{2k-1}\cdots\bar{b}_{2k},\;\ldots,\;\bar{b}_1\cdots\bar{b}_0\}=
(\bar{b}_{2k}\cdots\bar{b}_0),}
_{\aleph^{-1}((B_A'))=\aleph^{-1}((\bar{b}'_{2k}\cdots\bar{b}'_0))=\aleph^{-1}(\{\bar{b}'_{2k}\cdots\bar{b}'_0,\;\bar{b}'_{2k-1}\cdots\bar{b}'_{2k},\;\ldots,\;\bar{b}'_1\cdots\bar{b}'_0\})=
\{b'_0\cdots b'_{2k},\;b'_{2k}\cdots b'_{2k-1},\;\ldots,\;b'_1\cdots b'_0\}=(b'_0\cdots b'_{2k}),}$$
where $(b_0\cdots b_{2k})\in L_k/\pi$ and $(b'_0\cdots b'_{2k})\in L_{k+1}/\pi$. This establishes (i).

Every horizontal edge $v\aleph_\pi(v)$ of $M_k/\pi$ has $v\in L_k/\pi$ represented by  $\bar{b}_k\cdots \bar{b}_10b_1\cdots b_k$ in $L_k$, (so $v=(\bar{b}_k\cdots \bar{b}_10b_1\cdots b_k)$). There are $2^k$ such vertices in $L_k$ and at most $2^k$ corresponding vertices in $L_k/\pi$. For example, $(0^{k+1}1^k)$ and $(0(01)^k)$ are endpoints in $L_k/\pi$ of 2 horizontal edges of $M_k/\pi$, each. To prove that this implies (ii), we have to see that there cannot be more than 2 representatives $\bar{b}_k\cdots \bar{b}_1b_0b_1\cdots b_k$ and $\bar{c}_k\cdots \bar{c}_1c_0c_1\cdots c_k$ of a vertex $v\in L_k/\pi$, with $b_0=0=c_0$. Such a $v$ is expressible as $v=(d_0 \cdots b_0d_{i+1}\cdots d_{j-1}c_0\cdots d_{2k})$, with $b_0=d_i$, $c_0=d_j$ and $0<j-i\le k$. Let the substring $\sigma=d_{i+1}\cdots d_{j-1}$ be said $(j-i)$-{\it feasible}. Let us see that every $(j-i)$-feasible substring $\sigma$ forces in $L_k/\pi$ only vertices $\omega$ leading to 2 different (parallel) horizontal edges in $M_k/\pi$ incident to $v$. In fact, periodic continuation mod $n$ of $d_0\cdots d_{2k}$ both to the right of $d_j=c_0$ with minimal cyclic substring
$\bar{d}_{j-1}\cdots\bar{d}_{i+1}1d_{i+1}\cdots d_{j-1}0=P_r$ and to the left of $d_i=b_0$ with minimal cyclic substring $0d_{i+1}\cdots d_{j-1}1\bar{d}_{j-1}\cdots\bar{d}_{i+1}=P_\phi$ yields a 2-way infinite string that winds up onto a class $(d_0\cdots d_{2k})$ containing such an $\omega$. For example, some pairs of feasible substrings $\sigma$ and resulting vertices $\omega$ are: 
$$^{(\sigma,\omega)\;\;=\;\;(\emptyset,({\rm o}{\rm o}1)),\;\;(0,({\rm o}0{\rm o}11)),\;\;(1,({\rm o}1{\rm o})),\;\;(0^2,({\rm o}00{\rm o}111)),\;\;(01,({\rm o}01{\rm o}011)),\;\;(1^2,\,{\rm o}11{\rm o}0)),}_{(0^3,{\rm o}000{\rm o}1111)),\;\;
(010,({\rm o}010{\rm o}101101)),\;\;(01^2,({\rm o}011{\rm o})),\;\;(101,({\rm o}101{\rm o})),\;\;
(1^3,({\rm o}111{\rm o}00)),}$$ with `o' replacing $b_0=0$ and $c_0=0$,
and where $k=\lfloor\frac{n}{2}\rfloor$ has successive values $k=1,2,1,3,3,2,4,5,2,2,3$.
If $\sigma$ is a feasible substring and $\bar{\sigma}$ is its complemented reversal via $\aleph$, then the possible symmetric substrings $P_\phi\sigma P_r$ about $\rm{o}\sigma\rm{o}=0\sigma 0$ in a vertex $v$ of $L_k/\pi$ are in order of ascending length:
$$\begin{array}{c}
^{\hspace*{2mm}0\sigma 0,}_{\bar{\sigma}0\sigma 0\bar{\sigma},}\\
^{\hspace*{2mm}1\bar{\sigma}0\sigma 0\bar{\sigma}1,}
_{\sigma 1\bar{\sigma}0\sigma 0\bar{\sigma}1\sigma,}\\
^{\hspace*{2mm}0\sigma 1\bar{\sigma}0\sigma 0\bar{\sigma}1\sigma 0,}
_{\bar{\sigma}0\sigma 1\bar{\sigma}0\sigma 0\bar{\sigma}1\sigma 0\bar{\sigma},}\\
^{1\bar{\sigma}0\sigma 1\bar{\sigma}0\sigma 0\bar{\sigma}1\sigma 0\bar{\sigma}1,}
_{\cdots\cdots\cdots\cdots\cdots\cdots\cdots\cdots\cdots\cdots,}\\
\end{array}$$
\noindent where we use again `0' instead of `o' for the entries immediately preceding and following the shown central copy of $\sigma$. The lateral periods of $P_r$ and $P_\phi$ determine each one horizontal edge at $v$ in $M_k/\pi$ up to returning to $b_0$ or $c_0$,
so no entry $e_0=0$ of $(d_0\cdots d_{2k})$ other than $b_0$ or
$c_0$ happens such that $(d_0\cdots d_{2k})$ has a third representative
$\bar{e}_k\cdots \bar{e}_10e_1\cdots e_k$ (besides $\bar{b}_k\cdots
\bar{b}_10b_1\cdots b_k$ and $\bar{c}_k\cdots \bar{c}_10c_1\cdots
c_k$). Thus, those 2 horizontal edges are produced solely from the feasible
substrings $d_{i+1}\cdots d_{j-1}$ characterized above.
\end{proof}

To illustrate Theorem~\ref{thm3}, let $1<h<n$ in $\mathbb{Z}$ be such that $\gcd(h,n)=1$ and let $\lambda_h:L_k/\pi\rightarrow L_k/\pi$ be given by $\lambda_h((a_0a_1\cdots a_n))\rightarrow(a_0a_ha_{2h}\cdots a_{n-2h}a_{n-h}).$
For each such $h\le k$, there is at least one $h$-feasible substring $\sigma$ and a resulting associated vertex
$v\in L_k/\pi$ as in the proof of Theorem~\ref{thm3}. For example, starting at $v=(0^{k+1}1^k)\in L_k/\pi$ and applying $\lambda_h$ repeatedly produces a number of such vertices $v\in L_k/\pi$. If we assume $h=2h'$ with $h'\in\mathbb{Z}$,
then an $h$-feasible substring $\sigma$ has the form $\sigma=\bar{a}_1\cdots\bar{a}_{h'}a_{h'}\cdots a_1$, so there are at least $2^{h'}=2^{\frac{h}{2}}$ such $h$-feasible substrings.

\section{Dihedral Quotients}\label{s7}

An {\it involution} of a graph $G$ is a graph map $\aleph:G\rightarrow G$ such that $\aleph^2$ is the identity.
If $G$ has an involution, an {\it $\aleph$-folding} of $G$ is a graph $H$, possibly with loops, whose vertices $v'$ and edges or loops $e'$ are respectively of the form $v'=\{v,\aleph(v)\}$ and $e'=\{e,\aleph(e)\}$, where $v\in V(G)$ and $e\in E(G)$; $e$ has endvertices $v$ and $\aleph(v)$ if and only if $\{e,\aleph(e)\}$ is a loop of $G$.

Note that both maps $\aleph:M_k\rightarrow M_k$ and $\aleph_\pi:M_k/\pi\rightarrow M_k/\pi$ in Section~\ref{s6} are involutions.
Let $\langle B_A\rangle$ denote each horizontal pair $\{(B_A),\aleph_\pi((B_A))\}$ (as in Theorem~\ref{thm3}) of $M_k/\pi$, where $|A|=k$.
An $\aleph$-folding $R_k$ of $M_k/\pi$ is obtained whose vertices are the
pairs $\langle B_A\rangle$ and having:

{\bf (1)} an edge $\langle B_A\rangle\langle B_{A'}\rangle$ per skew-edge pair $\{(B_A)\aleph_\pi((B_{A'})),(B_{A'})\aleph_\pi((B_A))\};$

{\bf (2)} a loop at $\langle B_A\rangle$ per horizontal edge
$(B_A)\aleph_\pi((B_A))$; because of Theorem~\ref{thm3}, there may

\hspace*{7mm} be up to 2 loops at each vertex of $R_k$.

\begin{theorem}\label{th4} $R_k$ is a quotient graph of $M_k$ under an action
 $\Upsilon:D_{2n}\times M_k\rightarrow M_k$.\end{theorem}

\begin{figure}[htp]
\hspace*{26.5mm}
\includegraphics[scale=0.2]{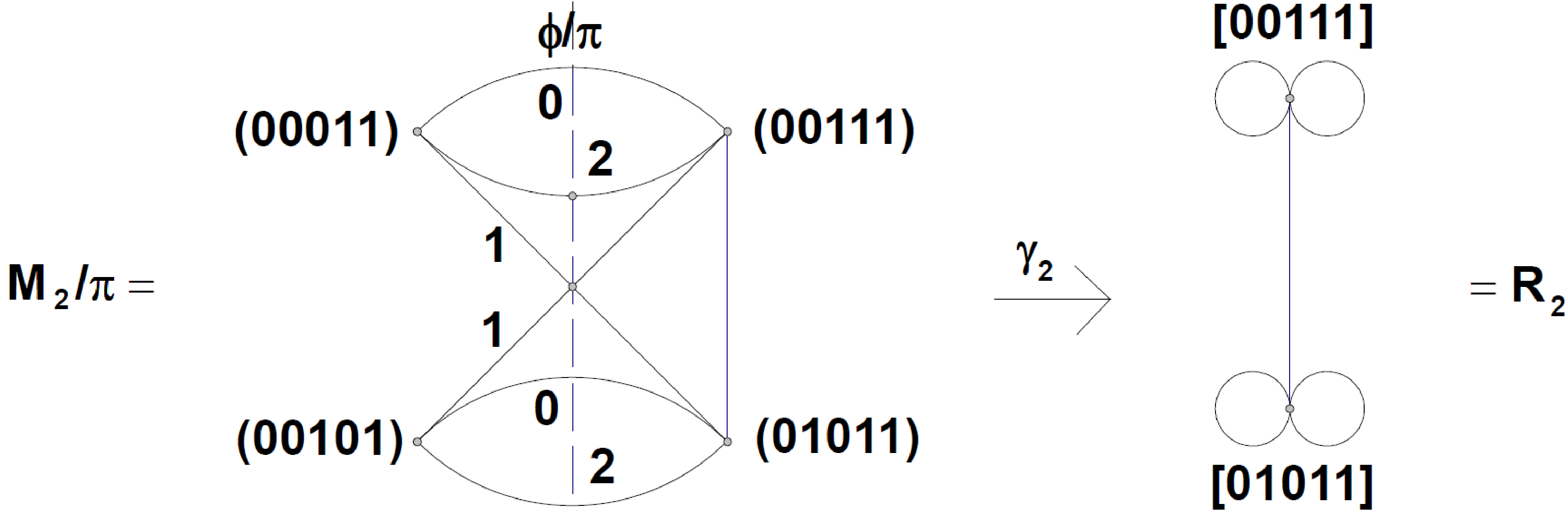}
\caption{Reflection symmetry of $M_2/\pi$ about a line $\phi/\pi$ and resulting graph map $\gamma_2$}
\end{figure}

\begin{proof} 
$D_{2n}$ is the semidirect product $\mathbb{Z}_n\rtimes_\varrho\mathbb{Z}_2$ via the group homomorphism $\varrho:\mathbb{Z}_2\rightarrow\mbox{Aut}(\mathbb{Z}_n)$, where $\varrho(0)$ is the identity and $\varrho(1)$ is the automorphism $i\rightarrow(n-i)$, $\forall i\in\mathbb{Z}_n$. If $*:D_{2n}\times D_{2n}\rightarrow D_{2n}$ indicates group multiplication and $i_1,i_2\in\mathbb{Z}_n$, then  $(i_1,0)*(i_2,j)=(i_1+i_2,j)$ and
$(i_1,1)*(i_2,j)=(i_1-i_2,\bar{j})$, for $j\in\mathbb{Z}_2$.
Set $\Upsilon((i,j),v)=\Upsilon'(i,\aleph^j(v))$, $\forall i\in\mathbb{Z}_n, \forall j\in\mathbb{Z}_2$, where $\Upsilon'$ is as in display~(\ref{d3}).
Then, $\Upsilon$ is a well-defined $D_{2n}$-action on $M_k$. By writing $(i,j)\cdot v=\Upsilon((i,j),v)$ and $v=a_0\cdots a_{2k}$, we have
$(i,0)\cdot v=a_{n-i+1}\cdots a_{2k}a_0\cdots a_{n-i}=v'$ and $(0,1)\cdot v'=\bar{a}_{i-1}\cdots \bar{a}_0\bar{a}_{2k}\cdots\bar{a}_i=(n-i,1)\cdot v=((0,1)*(i,0))\cdot v$, leading to the compatibility condition $((i,j)*(i',j'))\cdot v=(i,j)\cdot ((i',j')\cdot v)$.
\end{proof}

Theorem~\ref{th4} yields a graph projection $\gamma_k:M_k/\pi\rightarrow R_k$ for the action $\Upsilon$, given for $k=2$ in Figure 1. In fact, $\gamma_2$ is
associated with reflection of $M_2/\pi$ about the dashed vertical symmetry axis $\phi/\pi$ so that $R_2$ (containing 2 vertices and one edge between them, with each vertex incident to 2 loops) is given as its image. Both the representations of $M_2/\pi$ and $R_2$ in the figure have their edges indicated with colors 0,1,2, as arising inSection~\ref{s8}.

 \section{Lexical Procedure}\label{s8}

Let $P_{k+1}$ be the subgraph of the unit-distance graph of $\mathbb{R}$ (the real line) induced by the set $[k+1]=\{0,\ldots,k\}$.
We draw the grid $\Gamma=P_{k+1}\square P_{k+1}$ in the plane $\mathbb{R}^2$ with a diagonal $\partial$ traced from the lower-left vertex $(0,0)$ to the upper-right vertex $(k,k)$.
For each $v\in L_k/\pi$, there are $k+1$ $n$-tuples of the form
$b_0b_1\cdots b_{n-1}=0b_1\cdots b_{n-1}$ that represent $v$ with
$b_0=0$. For each such $n$-tuple, we construct a $2k$-path $D$ in $\Gamma$ from $(0,0)$ to $(k,k)$ in $2k$ steps indexed from $i=0$ to $i=2k-1$. This leads to a lexical edge-coloring implicit in \cite{KT}; see the following statement and Figure 2 (Section~\ref{s9}), containing examples of such a $2k$-path $D$ in thick trace.

\begin{theorem}\label{kitr} \cite{KT}
Each $v\in L_k/\pi$ has its $k+1$ incident edges assigned colors $0,1,\ldots,k$ by means of the following ``Lexical Procedure'', where $0\le i\in\mathbb{Z}$, $w\in V(\Gamma)$ and $D$ is a path in $\Gamma$.
Initially, let $i=0$, $w=(0,0)$ and $D$ contain solely the vertex $w$. Repeat $2k$ times the following sequence of steps (1)-(3), and then perform once the final steps (4)-(5):

\noindent{\bf (1)} If $b_i=0$, then set $w':=w+(1,0)$; otherwise, set $w':=w+(0,1)$.

\noindent{\bf (2)} Reset $V(D):=v(D)\cup\{w'\}$, $E(D):=E(D)\cup\{ww'\}$, $i:=i+1$ and $w:=w'$.

\noindent{\bf (3)} If $w\ne(k,k)$, or equivalently, if $i<2k$, then go back to step (1).

\noindent{\bf (4)} Set $\check{v}\in L_{k+1}/\pi$ to be the vertex of $M_k/\pi$ adjacent to
$v$ and obtained from its representative

$n$-tuple $b_0b_1\cdots$ $b_{n-1}=0b_1\cdots
b_{n-1}$ by replacing the entry $b_0$ by $\bar{b}_0=1$ in $\check{v}$, keeping the

entries $b_i$ of $v$ unchanged in $\check{v}$ for $i>0$.

\noindent{\bf (5)}
Set the color of the
edge $v\check{v}$ to be the number $c$ of horizontal (alternatively, vertical) arcs

of $D$ above $\partial$.
\end{theorem}

\begin{proof} 
If addition and subtraction in $[n]$ are taken modulo $n$ and we write $[y,x)=\{y,y+1,y+2,\ldots,x-1\}$, for $x,y\in[n]$, and $S^c=[n]\setminus S$, for $S=\{i\in[n]:b_i=1\}\subseteq[n]$, then the cardinalities of the sets $\{y\in S^c\setminus x: |[y,x)\cap S|<|[y,x)\cap S^c|\}$ yield all the edge
 colors, where $x\in S^c$ varies.\end{proof}

The Lexical Procedure of Theorem~\ref{kitr} yields a 1-fac\-tor\-i\-za\-tion not only for $M_k/\pi$ but also for $R_k$ and $M_k$. This is clarified by the end of Section~\ref{s9}.

\section{Uncastling and Lexical 1-Factorization}\label{s9}

\begin{figure}[htp]
\includegraphics[scale=0.275]{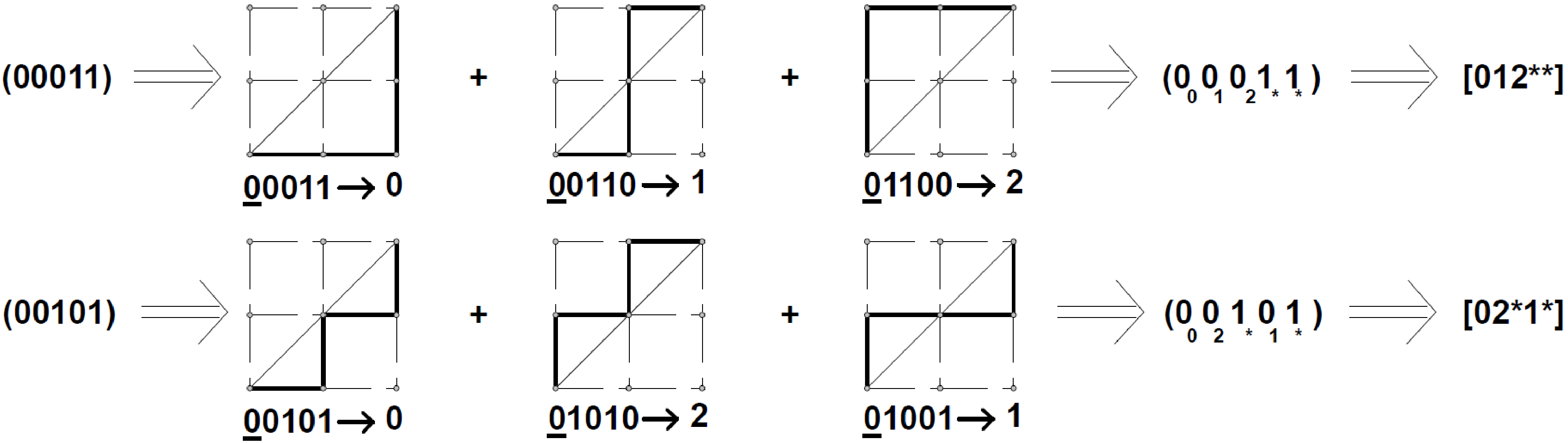}
\caption{Representing lexical-color assignment for $k=2$}
\end{figure}

A notation $\delta(v)$ is assigned to each pair $\{v,\aleph_\pi(v)\}\in R_k$, where $v\in L_k/\pi$, so that there is a unique $k$-germ $\alpha=\alpha(v)$ with $\langle F(\alpha)\rangle=\delta(v)$, (where the notation $\langle\cdot\rangle$ is as in $\langle B_A\rangle$ in Section~\ref{s7}). We exemplify $\delta(v)$ for $k=2$ in Figure 2, with the Lexical Procedure (indicated by arrows ``$\Rightarrow$'') departing from $v=(00011)$ (top) and $v=(00101)$ (bottom), passing to sketches of $\Gamma$ (separated by symbols ``+''), one  sketch (in which to trace the edges of $D\subset\Gamma$ as in Theorem~\ref{kitr})  per representative $b_0b_1\cdots b_{n-1}=0b_1\cdots b_{n-1}$ of $v$ shown under the sketch (where $b_0=0$ is underscored) and pointing via an arrow ``$\rightarrow$'' to the corresponding color $c\in[k+1]$. Recall this $c$ is the number of horizontal arcs of $D$ below $\partial$.

In each of the 2 cases in Figure 2 (top, bottom), an arrow ``$\Rightarrow$'' to the right of the sketches points to a modification $\hat{v}$ of $b_0b_1\cdots b_{n-1}=0b_1\cdots b_{n-1}$ obtained by setting as a subindex of each 0 (resp. 1) its associated color $c$ (resp. an asterisk ``$*$'' ). Further to the right, a third arrow ``$\Rightarrow$'' points to the $n$-tuple $\delta(v)$ formed by the string of subindexes of entries of $\hat{v}$ in the order they appear from left to right.

\begin{theorem}\label{un} Let $\alpha(v^0)=a_{k-1}\cdots a_1=00\cdots 0$.
To each $\delta(v)$ corresponds a sole $k$-germ $\alpha=\alpha(v)$ with $\langle F(\alpha)\rangle=\delta(v)$ by means of the following ``Uncastling Procedure'':
Given $v\in L_k/\pi$, let $W^i=01\cdots i$ be the maximal initial numeric (i.e., colored) substring of $\delta(v)$, so the length of $W^i$ is $i+1$ ($0\le i\le k$).
If $i=k$, let $\alpha(v)=\alpha(v^0)$; else, set $m=0$ and:

{\bf 1.} set $\delta(v^m)=\langle W^i|X|Y|Z^i\rangle$, where $ Z^i$ is the terminal $j_m$-substring of $\delta(v^m)$, with $j_m=$

\hspace*{5mm} $i+1$, and let $X,Y$ (in that order) start at contiguous numbers $\Omega$ and $\Omega-1\ge i$;

{\bf 2.} set $\delta(v^{m+1})=\langle W^i|Y|X|Z^i\rangle$;

{\bf 3.} obtain $\alpha(v^{m+1})$ from $\alpha(v^m)$ by increasing its entry $a_{j_m}$ by 1;

{\bf 4.} if $\delta(v^{m+1})=[01\cdots k*\cdots*]$, then  stop; else, increase $m$ by 1 and go to step 1.\end{theorem}

\begin{proof} This is a procedure inverse to that of Castling (Section~\ref{s3}), so 1-4 follow.\end{proof}

Theorem~\ref{un} allows to produce a finite sequence $\delta(v^0),\delta(v^1),\ldots,\delta(v^m),\ldots,$ $\delta(v^s)$ of $n$-strings  with $j_0\ge j_1\ge\cdots\ge j_m\cdots\ge j_{s-1}$ as in steps 1-4, and $k$-germs $\alpha(v^0),\alpha(v^1),\ldots,$ $\alpha(v^m),\ldots,$ $\alpha(v^s)$, taking from $\alpha(v^0)$ through the $k$-germs $\alpha(v^m)$, ($m=1,\ldots,s-1$), up to $\alpha(v)=\alpha(v^s)$ via unit incrementation of $a_{j_m}$, for $0\le m<s$, where each incrementation yields the corresponding $\alpha(v^{m+1})$. Recall $F$ is a bijection from the set $V({\mathcal T}_k)$ of $k$-germs onto $V(R_k)$, both sets being of cardinality $C_k$. Thus, to deal with $V(R_k)$ it is enough to deal with $V({\mathcal T}_k)$, a fact useful in interpreting Theorem~\ref{thm4} below.
For example $\delta(v^0)=\langle 04*3*2*1\,*\rangle=\langle 0|4*|3*2*1|*\rangle=\langle W^0|X|Y|Z^0\rangle$
with $m=0$ and $\alpha(v^0)=123$, continued in Table III with $\delta(v^1)=\langle W^0|Y|X|Z^0\rangle$, finally arriving to $\alpha(v^s)=\alpha(v^6)=000$.

\begin{center} TABLE III
$$\begin{array}{|c|ccccccc|c|c|}\hline
^{j_0=0}_{j_1=0}&^{\delta(v^1)}_{\delta(v^2)}&^{=}_{=}&^{\langle 0|3*2*1|4*|*\rangle}_{\langle 0|2*14*|3*|*\rangle}&^{=}_{=}&^{\langle 03*2*14**\rangle}_{\langle 02*14*3**\rangle}&^{=}_{=}&^{\langle 0|3*|2*14*|*\rangle}_{\langle 0|2*|14*3*|*\rangle}&^{\alpha(v^1)=122}_{\alpha(v^2)=121}&^{\langle F(122)\rangle=\delta(v^1)}_{\langle F(121)\rangle=\delta(v^2)}\\
^{j_2=0}_{j_3=1}&^{\delta(v^3)}_{\delta(v^4)}&^{=}_{=}&^{\langle 0|14*3*|2*|*\rangle}_{\langle 01|3*2|4*|**\rangle}&^{=}_{=}&^{\langle 014*3*2**\rangle}_{\langle 013*24***\rangle }&^{=}_{=}&^{\langle 01|4*|3*2|**\rangle}_{\langle 01|3*|24*|**\rangle}&^{\alpha(v^3)=120}_{\alpha(v^4)=110}&^{\langle F(120)\rangle=\delta(v^3)}_{\langle F(110)\rangle=\delta(v^4)}\\
^{j_4=1}_{j_5=2}&^{\delta(v^5)}_{\delta(v^6)}&^{=}_{=}&^{\langle 01|24*|3*|**\rangle}_{\langle 012|3|4*|***\rangle}&^{=}_{=}&^{\langle 0124*3***\rangle}_{\langle 01234****\rangle }&^{=}    &^{\langle 012\,|\,4\,*|3|**\rangle}&^{\alpha(v^5)=100}_{\alpha(v^6)=000}&^{\langle F(100)\rangle=\delta(v^5)}_{\langle F(000)\rangle=\delta(v^6)}\\\hline
\end{array}$$
\end{center}

A pair of skew edges $(B_A)\aleph_\pi((B_{A'}))$ and $(B_{A'})\aleph((B_A))$
in $M_k/\pi$, to be called a {\it skew reflection edge pair} ({\it SREP}), provides a color notation for any $v\in
L_{k+1}/\pi$ such that in each particular edge class mod $\pi$:

{\bf (I)} all edges receive a common color in $[k+1]$ regardless of the endpoint
on which the

\hspace*{6mm} Lexical Procedure (or
its modification immediately below) for $v\in L_{k+1}/\pi$ is applied;

{\bf (II)} the 2 edges in each SREP in $M_k/\pi$ are assigned a common color in $[k+1]$.

\noindent The modification in step (I) consists in replacing in Figure 2 each
$v$ by $\aleph_\pi(v)$ so that on the left we have instead now $(00111)$ (top) and $(01011)$ (bottom) with respective sketch subtitles

$$\begin{array}{ccc}
^{0011\underline{1}\rightarrow 0,}_{0101\underline{1}\rightarrow
0,}& ^{1001\underline{1}\rightarrow
1,}_{1010\underline{1}\rightarrow 2,}&
^{1100\underline{1}\rightarrow 2,}_{0110\underline{1}\rightarrow 1,}\\
\end{array}$$

\noindent resulting in similar sketches when the steps (1)-(5) of the
Lexical Procedure are taken with right-to-left reading and processing of the entries on the left side of the subtitles (before the arrows ``$\rightarrow$''), where the values of each $b_i$ must be taken complemented, (i.e., as $\bar{b_i}$).

Since an SREP in $M_k$ determines a unique edge $\epsilon$
of $R_k$ (and vice versa), the color received by the SREP can
be attributed to $\epsilon$, too. Clearly, each vertex of either
$M_k$ or $M_k/\pi$ or $R_k$ defines a bijection from its incident
edges onto the color set $[k+1]$. The edges obtained via $\aleph$ or $\aleph_\pi$
from these edges have the same corresponding colors.

\begin{theorem}\label{thm4}
A 1-factorization of $M_k/\pi$ by the colors
$0,1,\ldots,k$ is obtained via the Lexical Procedure that can be lifted to a covering 1-factorization of
$M_k$ and subsequently collapsed onto a folding 1-factorization of $R_k$. This insures the notation $\delta(v)$ for each $v\in V(R_k)$ so that there is a unique $k$-germ $\alpha=\alpha(v)$ with $\langle F(\alpha)\rangle=\delta(v)$.
\end{theorem}

\begin{proof}
As pointed out in (II) above, each SREP in $M_k/\pi$ has its edges with a common color in $[k+1]$. Thus, the $[k+1]$-coloring of $M_k/\pi$ induces a well-defined $[k+1]$-coloring of $R_k$. This yields the claimed collapsing to a folding 1-factorization of $R_k$. The lifting to a covering
1-factorization in $M_k$ is immediate. The arguments above determine that the collapsing 1-factorization in $R_k$ induces the claimed $k$-germs $\alpha(v)$.
\end{proof}

\section{Union of Lexical 1-Factors of Colors 0 and 1} \label{2fact}

Given a $k$-germ $\alpha$, let $(\alpha)$ represent the dihedral class $\delta(v)=\langle F(\alpha)\rangle$ with $v\in L_k/\pi$.
Let $W_{01}^k$ be the 2-factor given by the union of the 1-factors of colors 0 and 1 in $M_k$ (namely those formed by lifting the edges $\alpha\alpha^0$ and $\alpha\alpha^1$ of $R_k$ in the notation of \ Section~\ref{s10}, instead of those of colors $k$ and $k-1$, as in \cite{gmn}). The cycles of $W_{01}^k$ are grown in this section from specific paths, as suggested in Figure 3 for $k=2,3,4$ (say: cycle $C_0$ that starts with vertically expressed path $X(0)$, for $k=2$; cycles $C_0,C_1$ that start with vertically expressed paths $X(0),X(1)$, for $k=3$; and cycles $C_0,C_1,C_2$ that start with vertically expressed paths $X(0),X(1),X(2)$, for $k=4$). Here, vertices $v\in L_k$ (resp. $v\in L_{k+1}$) will be represented with: 

{\bf(a)} 0- (resp. 1-) entries replaced by their respective colors $0,1,\ldots,k$ (resp. asterisks); 

{\bf(b)} 1- (resp. 0-) entries replaced by asterisks
(resp. their respective colors $0,1,\ldots,k$);

 {\bf(c)} delimiting chevron symbols "$>$" or ``$\rangle$'' (resp. "$<$" or ``$\langle$''), instead of parentheses or 

 \hspace*{6mm} 
 brackets, indicating the reading direction of {\it ``forward''} (resp. {\it ``backward''}) $n$-tuples. 

\begin{figure}[htp] 
\includegraphics[scale=0.74]{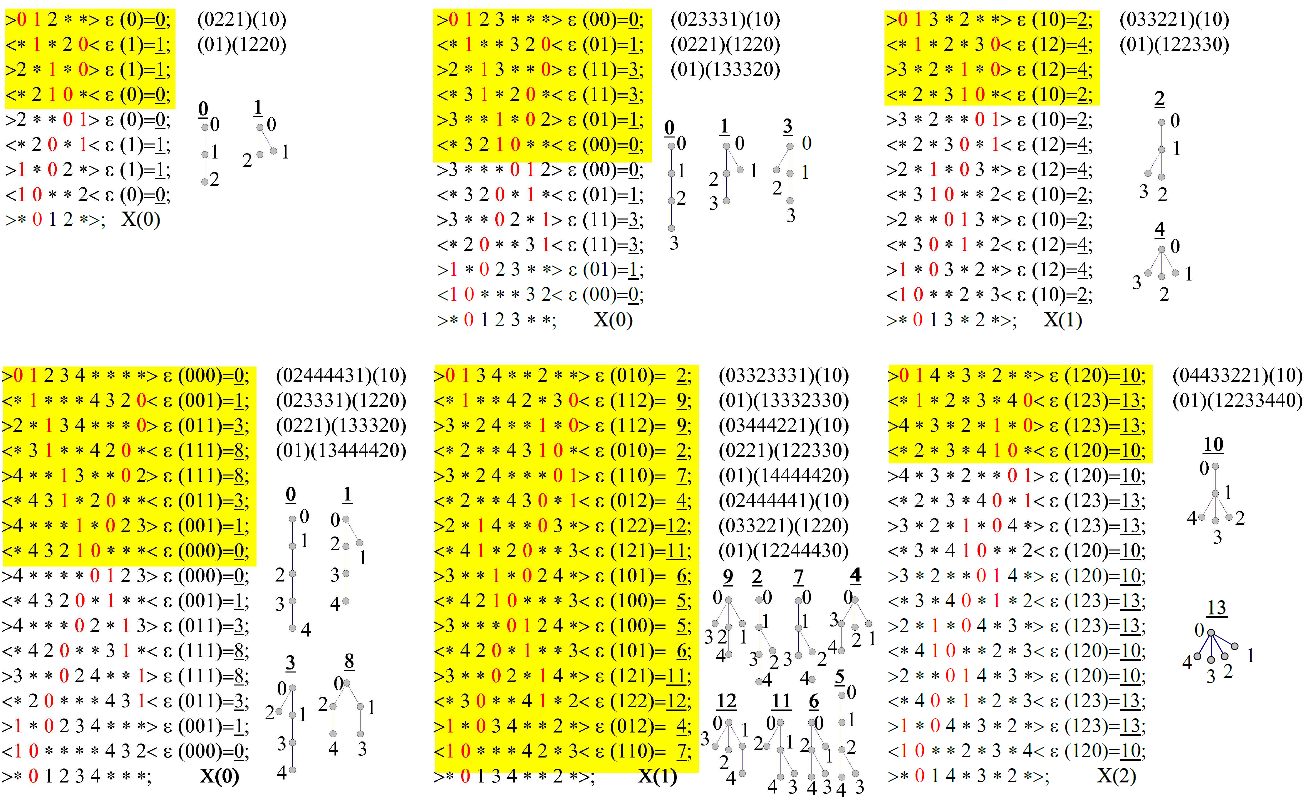}
\caption{Cycles of $W_{01}^k$ in $M_k$, ($k=2,3,4$)}
\end{figure}

\noindent Each such vertex $v$ is shown to belong (via the set membership symbol expressed by ``$\varepsilon$'' in Figure 3) to $(\alpha_v)$, where $\alpha_v$ is the $k$-germ of $v$, also expressed as its (underlined decimal) natural order. In each case, Figure 3 shows a vertically presented path $X(i)$ of
length $4k-1$ in the corresponding cycle $C_i$ starting at the vertex $w=b_0b_1\cdots b_{2k}=01\cdots *$ of smallest natural order and proceeding by traversing the edges colored 1 and 0, alternatively. The terminal vertex of such subpath is $b_{2k}b_0b_1\cdots b_{2k-1}=*01\cdots b_{2k-1}$, obtained by translation mod $n$ from $w$.  

\begin{obs}\label{ha}
The initial entries of the vertices in each $C_i$ are presented downward, first in the 0-column of $X(i)$, then in the $(2k-j)$-column of $X(i)$, ($j\in[2k]$, only up to $|C_i|$.
\end{obs}

In Figure 3,  initial entries are red if they are in $\{0,1\}$ and
each cycle $C_i$ is encoded on its top right by a vertical sequence of expressions $(0\cdots 1)(1\cdots 0)$ that allows to get the sequence of initial entries of the succeeding vertices of $C_i$ by interspersing asterisks between each 2 terms inside parentheses (),
then removing those ().    
An ordered tree $T=T_v$ (Remark~\ref{re}) for each $v$ in the exemplified $C_i(=C$ in Proposition 2(v) \cite{gmn}) is shown at the lower right of its case in Figure 3. 
Each of these $T_v$ for a specific $C_i$ corresponds to the $k$-germ $\alpha_v$ (so we write $F(\alpha_v)=F(T_{\alpha_v})$) and is headed in the figure by its (underlined decimal) natural order. In the figure, vertices of each $T_v$ are denoted $i$, instead of $v_i$ ($i\in[k+1]$).
The trees corresponding to the $k$-germs in each case are obtained by applying {\it root rotation} \cite{gmn}, consisting in replacing the tree root by its leftmost child and redrawing the ordered tree accordingly.  
 A {\it plane tree} is an equivalence relation of ordered trees under root rotations. In the notation of Section~\ref{s10}, applying one root rotation has the same effect as traversing first an edge $\alpha\alpha_0$ in $C_i$ and then edge $\beta\beta_1$, also in $C_i$, where $\beta=\alpha_0$.  
 In each case, a yellow box shows a subpath of $X(i)$ with $\frac{1}{n}|V(C_i)|$ vertices of $C_i$ that takes into account the rotation symmetry of the associated plane tree ${\mathcal T}_i$.
In Figure 3 for $k=4$, successive application of root rotations on the second cycle, $C_1$, produces the cycle $(\underline{9},\underline{2},\underline{4},\underline{11},\underline{5},\underline{6},\underline{12},\underline{7})$, the square graph of $C_1$,  that starts downward from the second row or upward from the third row, thus covering respectively the vertices of $L_4$ or $L_5$ in the class.
 Let $D_i$ ($i\ge  0$) be the set of substrings of length $2i$ in an $F(\alpha)$ with exactly $i$ color-entries such that in every prefix (i.e. initial substring), the number of asterisk-entries is at least as large as the number of color-entries. The elements of $D=\cup_{i\ge 0}D_i$ are known as {\it Dyck words}. 
Each $F(\alpha)$ is of the form $0v1u*$, where $u$ and $v$ are Dyck words and 0 and 1 are colors in $[k+1]$ \cite{gmn}. 
The ``forward'' $n$-tuple  $F(\alpha)$  in $(\underline{9},\underline{2},\underline{4},\underline{11},\underline{5},\underline{6},\underline{12},\underline{7})$ can be written with parentheses enclosing such Dyck words $v$ and $u$, namely  $0()1(34**2*)*$, for $\underline{2}$; $0(3*24**)1()*$, for $\underline{9}$; $0()1(3*24**)*$, for $\underline{7}$; $0(3*2*)1(4*)*$, for $\underline{12}$; $0(24*3**)1()*$, for $\underline{6}$; $0()1(24*3**)*$, for $\underline{5}$; $0(2*)1(4*3*)*$, for $\underline{11}$; and $0(34**2*)1()*$, for $\underline{4}$. Similar treatment holds for ``backward'' $n$-tuples.
As in Figure 3, each pair $(0\cdots 1)(1\cdots 0)$ represents 2 paths in the corresponding cycle,  of lengths $2|(0\cdots 1)|-1$ and $2|(1\cdots 0)|$,\ adding  up to $4k+2$. If these 2 paths are of the form $0v1u*$ and $0v'1u'*$ (this one read in reverse), then $|u|+2=|(0\cdots 1)|$ and $|u'|+2=|(1\cdots 0)|$. 
Reading these paths starts at a 0-entry and ends at a 1-entry. In reality, the collections of paths obtained from the 1-factors here have the leftmost entry of the $n$-tuples representing their vertices constantly equal to $1\in\mathbb{Z}_2$ before taking into consideration (items (a) and (b) above, but with the reading orientations given in item (c). 

The 1-factor of color 0 makes the endvertices of each of its edges to have their representative plane trees obtained from each other by {\it horizontal reflection} $\Phi=F\alpha^0F^{-1}$. For example, Figure 3 shows that for $k=2$: both $\underline{0},\underline{1}$ in $X(0)$ are fixed via $\Phi$; for $k=3$: $\underline{0}$ in $X(0)$ is fixed via $\Phi$ and $\underline{1},\underline{3}$ in $X(0)$ correspond to each other via $\Phi$; and $\underline{2},\underline{4}$ in $X(1)$ are fixed via $\Phi$; for $k=4$: $\underline{0},\underline{8}$ in $X(0)$ are fixed via $\Phi$ and $\underline{1},\underline{3}$ in $X(0)$ correspond to each other via $\Phi$; $\underline{5},\underline{9}$ in $X(1)$ are fixed via $\Phi$ and the pairs $(\underline{5},\underline{9})$, $(\underline{2},\underline{7})$, $(\underline{4},\underline{12})$ and $(\underline{6},\underline{11})$ in $X(1)$ are pairs of correspondent plane trees via $\Phi$; and $\underline{10},\underline{13}$ in $X(2)$ are fixed via $\Phi$. This horizontal reflection symmetry arises from Theorem~\ref{thm3}. It accounts for each pair of contiguous rows in any $X(i)$ corresponding to a 0-colored edge.                                      
For $k=5$, this symmetry via $\Phi$ occurs in all cycles $C_i$ ($i\in[6]$). But we also have $F(\underline{22})=\rangle 024**135***\rangle$, where $\underline{22}=(1111)$ in $C_0$ and $F(\underline{39})=\rangle 03*2*15*4**\rangle$, where $\underline{39}=(1232)$ in $C_3$, both having their 1-colored edges leading to reversed reading between $L_5$ and $L_6$, again by Theorem~\ref{thm3}. Moreover, $F((11\cdots 1))$ has a similar property only if $k$ is odd; but if $k$ is even, a 0-colored edge takes place, instead of the 1-colored edge for $k$ odd.
These cases reflect the following lemma (which can alternatively be implied from Theorem~\ref{nt} (B)-(C)) via the correspondence $i\leftrightarrow k-i$, ($i\in[k+1]$).

\begin{obs}\label{lem} {\bf(A)} Every $0$-colored edge is represented via $\Phi=F\alpha^0F^{-1}$. {\bf(B)} Every 1-colored edge is represented via the composition $\Psi$ of $\Phi$ (first) and root rotation (second).\end{obs} 

By  Theorem~\ref{thm3}(ii), the number $\xi$ of contiguous pairs of vertices of $M_k$ in each $C_i$ with a common $k$-germ happens in pairs. The first cases for which this $\xi$  is null happens for $k=6$, namely for the 2 reflection pairs $(\rangle 012356**4****\rangle,\rangle 01235*46*****\rangle)$ and $(\rangle 01246*5**3***\rangle,\rangle 0124*36*5****\rangle)$ whose respective ordered trees are enantiomorphic, i.e they are reflection via $\Phi$ of each other. We say that these 2 cases are {\it enantiomorphic}. In fact, the presence of a pair of enantiomorphic ordered trees in a case of an $X(i)$ will be distinguished by saying that the case is {\it enantiomorphic}. For example, $k=4$ offers $(\underline{1},\underline{3})$ as the sole enantiomorphic pair in $X(0)$, and $(\underline{2},\underline{7})$, $(\underline{4},\underline{12})$, $(\underline{11},\underline{6})$ as all the enantiomorphic pairs in $X(1)$. 
Each enantiomorphic cycle $C_i$ or each cycle $C_i$ with $\xi=2$ has $|C_i|=2k(4k+2)$. If $\xi=2\zeta$ with $\zeta>1$, then $|C_i|=\frac{2k(4k+1)}{\zeta}$.
On account of these facts, we have the following:

\begin{obs}\label{thm12} For each integer $k>1$, there is a natural bijection $\Lambda$ from the $k$-edge plane trees onto the cycles of $W_{01}^k$, as well as a partition ${\mathcal P}_k$ of the $k$-germs (or the ordered trees they represent via $F$), with each class of ${\mathcal P}_k$ in natural correspondence either to a $k$-edge plane tree or to a pair of enantiomorphic $k$-edge plane trees disconnecting in $W_{01}^k$ the forward (in $L_k$) and reversed (in $L_{k+1}$) readings of each vertex $v$ of their associated cycles via $\Lambda$. 
\end{obs}

\section{Reinterpretation of the Middle-Levels Theorem}\label{ultimo}

Fo each cycle $C_i$ of $W_{01}^k$, the ordered trees of its plane tree ${\mathcal T}_i$ with leftmost subpath of length 1 from $v_0$ to a vertex $v_h$ determine 6-cycles touching $C_i$ in 2 nonadjacent edges as follows: 
Let $t_i<k$ be the number of degree 1 vertices of ${\mathcal T}_i$. Let $\tau_i$ be the number of rotation symmetries of ${\mathcal T}_i$. Then, there are $\frac{t_i}{\tau_i}$
 classes mod $n$ of pairs of vertices $u,v$ at distance 5 in $C_i$ with $u\in L_{k+1}$ ahead of $v\in L_k$ in $X(i)$ and associated color $h\in\{2,\ldots,k\}$ such that: 
{\bf(i)} $u$ and $v$  are adjacent via $h$; 
{\bf(ii)} $u$ (resp. $v$) has cyclic backward (resp. forward) reading $\langle\cdot\cdots * h 0 * \cdots\langle$ (resp. $\rangle\cdots * 0 h *\cdots\cdot\rangle$);
{\bf(iii)} the column in which the occurrences of $h$ in (ii) happen at distance 5 looks between 
$u$ and $v$ (both included) as the transpose of $(h,*,0,0,*,h)$. Recall there are $n$ such columns. 
In each case, the vertices $u'$ and $v'$ in $C_i$ preceding respectively $u$ and $v$ in $X_i$ are endvertices of a 3-path $u'u''v''v'$ in $M_k$ with the edge $u''v''$ in a cycle $C_j\ne C_i$ in $W_{01}^k$. 
The 6-cycle $U_i^j=(uu'u''v''v'v)$ has as symmetric difference $U_i^j\Delta(C_i\cup C_j)$ a cycle in $M_k$ whose vertex set is $V(C_i\cup C_j)$.
 With $u',u'',v'',v',v,u,u'$ shown vertically, Figure 4 illustrates $U_i^j$, twice each for $k=3,4$.  
That symmetric difference replaces respectively the edges $u''v''$, $v'v$, $uu'$ in $C_i\cup C_j$ by $u'u''$, $v''v''$, $vu$. In the figure, vertically contiguous positions holding a common number $g$ (meaning adjacency via color $g$) are presented in red if $g\in\{0,1\}$, e.g. $u''v''$ with $g=1$, in column say $r_1$ exactly at the position where $u$ and $v$ differ, (however having common color $h$ as in (i) above) and in orange, otherwise. The column $r_2$ (resp $r_3$) in each instance of Figure 3 containing color 1 in $u'$ (resp. 0 in $v'$) and a color $c\in\{2,\ldots,k\}$ in $v'$ (resp. color $d\in\{2,\ldots,k\}$ in $u'$) starts with $1,*,c,c$, (resp. $d,d,*,0$), where $c,d\in\{2,\ldots,k\}$. Then, $r_1,r_2,r_3$ are the only columns having changes in the binary version of $U_i^j$. All other columns have their first 4 entries alternating asterisks and colors. In the first disposition in item (ii) above, we have that: {\bf(ii')} $u''$ (resp. $v''$) has cyclic backward (resp. forward) reading $\langle\cdot\cdots d 1 0 * \cdots\langle$ (resp. $\rangle\cdot\cdots * 1 * 0\cdots\rangle$). 

In the previous paragraph, ``ahead'' can be replaced by ``behind'', yielding additional 6-cycles $U_i^j$ by modifying adequately the accompanying text. 

\begin{figure}[htp]\label{fig3}
\includegraphics[scale=0.282]{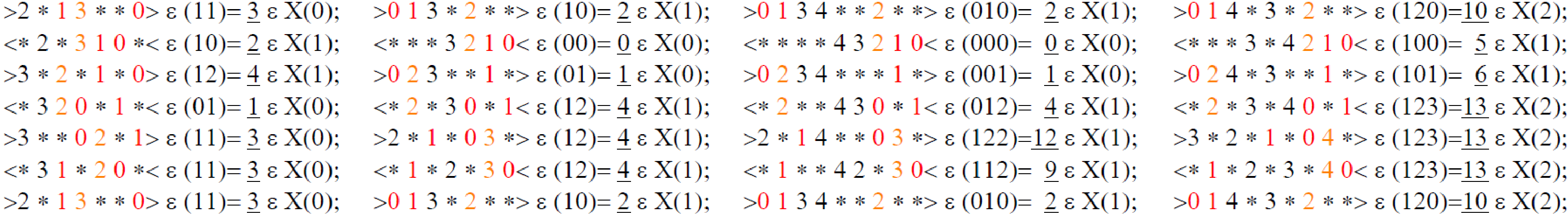}
\vspace*{-5mm}
\caption{Examples of 6-cycles $U_i^j$ for $k=4,5$}
\end{figure}
 
\begin{theorem}\label{thm23} \cite{M,gmn}
Let $0<k\in\mathbb{Z}$. A Hamilton cycle in $M_k$ is obtained by means of the symmetric differences of $W_{01}^k$ with the members of a set of pairwise edge-disjoint 6-cycles $U_i^j$.  
\end{theorem}

\begin{proof}  Clearly, the statement holds for $k=1$. Assume $k>1$. 
Let $\mathcal D$ be the digraph whose vertices are the cycles $C_i$ of $W_{01}^k$, with an arc from $C_i$ to $C_j$, for each 6-cycle $U_i^j$, where $C_i,C_j\in V(\mathcal D)$ with $i\ne j$. Since $M_k$ is connected, then $\mathcal D$ is connected. 
Moreover, the outdegree and indegree of every $C_i$ in $\mathcal D$ is $2n\frac{t_i}{\tau_i}$ (see items (ii) and (ii') above), in proportion with the length of $C_i$, a $\frac{1}{n}$-th of which is illustrated in each yellow box of Figure 3. 
Consider a spanning tree ${\mathcal D}'$ of $\mathcal D$. 
Since all vertices of $\mathcal D$ have outdegree $>0$, there is a ${\mathcal D}'$ in which the outdegree of each vertex is 1. This way, we avoid any pair of 6-cycles $U_i^j$ with common $C_i$ in which the associated distance-6 subpaths from $u'$ to $v$ in $C_i$ do not have edges in common.   
For each $a\in A({\mathcal D}')$, let $\nabla(a)$ be its associated 6-cycle $U_i^j$. Then, $\{\nabla(a);a\in A({\mathcal D}')\}$ can be selected as a collection of edge-disjoint 6-cycles. By performing all symmetric differences $\nabla(a)\Delta(C_i\cup C_j)$ corresponding to these 6-cycles, a Hamilton cycle is obtained.
\end{proof}

\section{Alternate Viewpoint on Ordered Trees}\label{rootk}

To have the viewpoint of \cite{gmn}, replace $v_0$ by $v_k$ as root of the ordered trees. We start with examples.
Figure 5 shows on its left-hand side the 14 ordered trees for $k=4$ encoded at the bottom of Table I. Each such tree $T=T(\alpha)$ is headed on top by its $k$-germ $\alpha$, in which the entry $i$ producing $T$ via Castling is in red. Such $T$ has its vertices denoted on their left and its edges denoted on their right, with their notation $v_i$ and $e_j$ given in  Section~\ref{2fact}. Castling here is indicated in any particular tree $T=T(\alpha)\ne T(00\cdots 0)$ by distinguishing in red the largest subtree common with that of the parent tree of $T$ (as in Theorem~\ref{thm1}) whose Castling reattachment produces $T$. This subtree corresponds with substring $X$ in Theorem~\ref{thm2}. In each case of such parent tree, the vertex in which the corresponding tree-surgery transformation leads to such a child tree $T$ is additionally labeled (on its right) with the expression  of its $k$-germ, in which the entry to be modified in the case is set in red color.

On the other hand, the 14 trees in the right-hand side of Figure 5 have their labels set by making the root to be $v_k$ (instead of $v_0$), then going downward to $v_0$ (instead of $v_k$) while gradually increasing  (instead of decreasing) a unit in the subindex $j$ of the denomination $v_j$, sibling by sibling from left to right at each level. The associated $k$-germ headers on this right-hand side of the figure correspond to the new root viewpoint. This determines a bijection $\Theta$ established by correspondence between the old and the new header $k$-germs. In our example, it yields an involution formed by the pairs $(001,100)$, $(011,110)$, $(120,012)$ and $(112,121)$, with fixed 000, 010, 101, 111, 122 and 123. 

\begin{figure}[htp]
\includegraphics[scale=0.278]{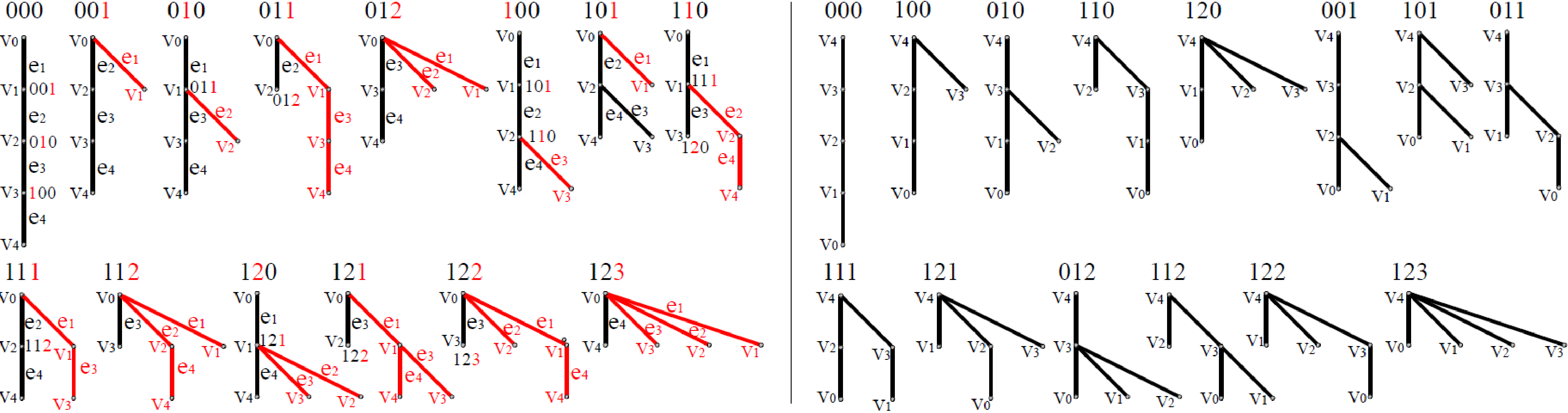}
\vspace*{-5mm}
\caption{Generation of ordered trees for $k=4$}
\end{figure}

The function $\Theta$ seen from the $k$-germ viewpoint, namely as the composition function $F^{-1}\Theta F$, behaves as follows. Let $\alpha=a_{k-1}a_{k-2}\cdots a_2a_1$ be a $k$-germ and let $a_i$ be the rightmost occurrence of its largest integer value. A substring $\beta$ of $\alpha$ is said to be an {\it atom} if it is either formed by a sole 0 or is a maximal strictly increasing substring of $\alpha$ not starting with 0. For example, consider the 17-germ $\alpha'=0123223442310121$. By enclosing the successive atoms between parentheses, $\alpha'$ can be written as $\alpha'=(0)123(2)(234)4(23)(1)(0)(12)(1)$, obtained by inserting in a {\it base string} $\gamma'=1\cdots a_i=1234$ all those atoms according to their order, where $\gamma'$ appears partitioned into subsequent un-parenthesized atoms distributed and interspered from left to right in $\alpha'$ just once each as further to the right as possible. 
This atom-parenthesizing procedure works for every $k$-germ $\alpha$ and determines a corresponding base string $\gamma$, like the $\gamma'$ in our example. 

\begin{theorem}\label{pro}
Given a $k$-germ $\alpha=a_{k-1}a_{k-2}\cdots a_2a_1$, let $a_i$ be the leftmost occurrence of its largest integer value. Then, $\alpha$ is obtained from a base string $\gamma=1\cdots a_i$ by inserting in $\gamma$ all atoms of $\alpha\setminus \gamma$ in their left-to-right order. Moreover, $F^{-1}\Theta F(\alpha)$ is obtained by reversing the insertion of those atoms in $\gamma$, in right-to-left fashion. For both insertions, $\gamma$ is partitioned into subsequent atoms distributed in $\alpha$ and $F^{-1}\Theta F(\alpha)$ as further to the right as possible. 
\end{theorem}

\begin{proof}
$F^{-1}\Theta F(\alpha)$ is obtained by reversing the position of the parenthesized atoms, inserting them between the substrings of a partition of $\gamma$ as the one above but for this reversing situation. In the above example of $\alpha'$, it is
$F^{-1}\Theta F(\alpha')=(1)(12)(0)(1)12(234)34(23)(2)(0)$. 
\end{proof}

\section{Appendix I: Germ Structure of 1-Factorizations}\label{s10}

We present each $c\in V(R_k)$ via the pair $\delta(v)=\{v,\aleph_\pi(v)\}\in R_k$ ($v\in L_k/\pi$)
of Section~\ref{s9} and via the $k$-germ $\alpha$ for which
$\delta(v)=\langle F(\alpha)\rangle$, and view $R_k$
as the graph whose vertices are the $k$-germs $\alpha$, with adjacency
inherited from that of their $\delta$-notation via $F^{-1}$ (i.e. Uncastling).
So, $V(R_k)$ is presented as in the natural ($k$-germ) enumeration (see Section~\ref{srgg}). 

\begin{center}TABLE IV
$$\begin{array}{||l||l||l||l|l|l|l||l|l|l|l||}
m&\alpha& F(\alpha) & F^3(\alpha) & F^2(\alpha) & F^1(\alpha) & F^0(\alpha) & \alpha^3 & \alpha^2 & \alpha^1 & \alpha^0\\\hline\hline
0&0 & 012** &-& 012** & 02*1* & 1\,2*\!*0   &-& 0 & 1 & 0\\
1&1 & 02*1* &-& 1*02* & 012** & 2\!*\!1\!*0 &-& 1 & 0 & 1\\\hline\hline
0&00\!&\!0123\!*\!**\!&\!0123*\!**\!&\!013\!*\!2\!*\!*\!&\!023\!**1*\!&\!123\!*\!*\!*0\!&\!00\!&\!10\!&\!01\!&\!00\\
1&01\!&\!023\!*\!*1*\!&\!1\!*\!023\!*\!*\!&\!1\!*\!03\!*\!2*\!&\!0123*\!**\!&\!2\!*\!13\!*\!*\,0\!&\!01\!&\!12\!&\!00\!&\!11\\
2&10\!&\!013\!*\!2\!*\!*\!&\!02\!*\!20\!*\!*\!&\!0123\!*\!**\!&\!03\!*\!2\!*\!1*\!&\!13\!*\!2\!*\!*\,0\!&\!11\!&\!00\!&\!12\!&\!10\\
3&11\!&\!02\!*\!13\!*\!*\!&\!013\!*\!2\!*\!*\!&\!13\!*\!*02*\!&\!02\!*\!13\!*\!*\!&\!10\!*\!*2\!*3\!&\!10\!&\!11\!&\!11\!&\!01\\
4&12\!&\!03\!*\!2\!*\!1*\!&\!2\!*\!1\!*\!03*\!&\!1\!*\!023\!*\!*\!&\!013\!*\!2\!*\!*\!&\!3\!*\!2\!*\!1\!*0\!&\!12\!&\!01\!&\!10\!&\!12\\\hline\hline
\end{array}$$\end{center}

To start with, examples of such presentation are shown in Table IV
for $k=2$ and $3$, where $m$, $\alpha=\alpha(m)$ and $F(\alpha)$ are shown in the first 3 columns, for $0\le m<C_k$. The neighbors of
$F(\alpha)$ are presented in the central columns of the table as $F^k(\alpha)$, $F^{k-1}(\alpha)$, $\ldots$, $F^0(\alpha)$
respectively for the edge colors $0,1,\ldots,k$, with notation given via the effect of
function $\aleph$. The last columns yield the $k$-germs $\alpha^k$, $\alpha^{k-1}$, $\ldots$, $\alpha^0$ associated via $F^{-1}$ respectively to the listed neighbors $F^k(\alpha)$, $F^{k-1}(\alpha)$ , $\ldots$, $F^0(\alpha)$ of $F(\alpha)$ in $R_k$.

\begin{center}TABLE V
$$\begin{array}{||c||c||c|c|c|c|c||c||c||c||c|c|c|c|c||}
^m_-&^\alpha_{--}&^{\alpha^4}_{--}&^{\alpha^3}_{--}&^{\alpha^2}_{--}&^{\alpha^1}_{--}&^{\alpha^0}_{--}&&^m_-&^\alpha_{--}&^{\alpha^4}_{--}&^{\alpha^3}_{--}&^{\alpha^2}_{--}&^{\alpha^1}_{--}&^{\alpha^0}_{--}\\
^0_1&^{000}_{001}&^{000}_{001}&^{100}_{101}&^{010}_{012}&^{001}_{000}&^{000}_{011}&&^{\;\; 7}_{\;\; 8}&^{110}_{111}&^{100}_{111}&^{111}_{110}&^{110}_{122}&^{012}_{011}&^{010}_{111}\\
^2_3&^{010}_{011}&^{011}_{010}&^{121}_{120}&^{000}_{011}&^{112}_{111}&^{110}_{001}&&^{\;\; 9}_{10}&^{112}_{120}&^{101}_{122}&^{122}_{011}&^{112}_{100}&^{010}_{123}&^{112}_{120}\\
^4_5&^{012}_{100}&^{012}_{110}&^{123}_{000}&^{001}_{120}&^{110}_{101}&^{122}_{100}&&^{11}_{12}&^{121}_{122}&^{121}_{120}&^{010}_{112}&^{121}_{111}&^{122}_{121}&^{101}_{012}\\
^6&^{101}&^{112}&^{001}&^{123}&^{100}&^{121}&&^{13}&^{123}&^{123}&^{012}&^{101}&^{120}&^{123}\\
^{-}&^{--}&^{--}_{3**}&^{--}_{***}&^{--}_{3**}&^{--}_{*2*}&^{--}_{**1}
&&
^{-}&^{--}&^{--}_{3**}&^{--}_{***}&^{--}_{3**}&^{--}_{*2*}&^{--}_{**1}
\end{array}$$\end{center}

For $k=4$ and 5, Tables V-VI, below, have a similar respective natural enumeration adjacency disposition.
We can generalize these tables directly to  {\it Colored Adjacency Tables} denoted {\it CAT}$(k)$, for $k>1$. This way, Theorem~\ref{nt}(A) below is obtained as
indicated in the aggregated last row upending Tables V-VI citing the only non-asterisk entry, for each of $i=k,k-2,\ldots,0$,  as a number $j=(k-1),\ldots,1$ that leads to entry equality in both columns $\alpha=a_{k-1}\cdots a_j\cdots a_1$ and $\alpha^i=a_{k-1}^i\cdots a_j^i\cdots a_1^i$, that is $a_j=a_j^i$. Other important properties are contained in the remaining items of Theorem~\ref{nt}, including (B), that the columns $\alpha^0$ in all CAT$(k)$, ($k>1$), yield an (infinte) integer sequence. 

\begin{theorem}\label{nt} Let: $k>1$, $j(\alpha^k)=k-1$ and $j(\alpha^{i-1})=i$, ($i=k-1,\ldots,1$). Then: 
{\bf(A)} each column $\alpha^{i-1}$ in {\rm CAT}$(k)$, for $i\in[k]\cup\{k+1\}$, preserves the respective $j(\alpha^{i-1})$-th entry of $\alpha$;
{\bf(B)} the columns $\alpha^k$ of all {\rm CAT}$(k)$'s for $k>1$ coincide into an RGS sequence and thus into an integer sequence ${\mathcal S}_0$, the first $C_k$ terms of which form an idempotent permutation for each $k$;
{\bf(C)} the integer sequence ${\mathcal S}_1$ given by concatenating the $m$-indexed intervals $[0,2),[2,5), \ldots,$ $[C_{k-1},C_k)$, etc. in column $\alpha^{k-1}$ of the corresponding tables {\rm CAT}$(2)$, {\rm CAT}$(3),\ldots$, {\rm CAT}$(k)$, etc. allows to encode all columns $\alpha^{k-1}$'s;
{\bf(D)} for each $k>1$, there is an idempotent permutation given in the $m$-indexed interval $[0,C_k)$ of the column $\alpha^{k-1}$ of {\rm CAT}$(k)$; such permutation equals the one given in the interval $[0,C_k)$ of the column $\alpha^{k-2}$ of {\rm CAT}$(k+1)$.
\end{theorem}

\begin{proof} (A) holds as a continuation of the observation made above with respect to the last aggregated row in Figures V-VI. Let $\alpha$ be a $k$-germ. Then $\alpha$ shares with $\alpha^k$ 
(e.g. the leftmost column $\alpha^i$ in Tables IV-VI, for $0\le i\le k$) all the entries to the left of the leftmost entry 1, which yields (B). Note that if $k=3$ then $m=2,3,4$ yield for $\alpha^{k-1}$ the idempotent permutation $(2,0)(4,1)$, illustrating (C). (D) can be proved similarly.\end{proof}

\begin{center} TABLE VI
$$\begin{array}{||c||c||c|c|c|c|c|c||c||c||c||c|c|c|c|c|c||c||}
^m_-&^\alpha_{--}&^{\alpha^5}_{--}&^{\alpha^4}_{--}&^{\alpha^3}_{--}&^{\alpha^2}_{--}&^{\alpha^1}_{--}&^{\alpha^0}_{--}&&^m_-&^\alpha_{--}&^{\alpha^5}_{--}&^{\alpha^4}_{--}&^{\alpha^3}_{--}&^{\alpha^2}_{--}&^{\alpha^1}_{--}&^{\alpha^0}_{--}\\
^{0}_{1}&^{0000}_{0001}&^{0000}_{0001}&^{1000}_{1001}&^{0100}_{0101}&^{0010}_{0012}&^{0001}_{0000}&^{0000}_{0011}&&^{21}_{22}&^{1110}_{1111}&^{1111}_{1110}&^{1100}_{1111}&^{1221}_{1220}&^{0110}_{0122}&^{1112}_{1111}&^{1110}_{0111}\\
^{2}_{3}&^{0010}_{0011}&^{0011}_{0010}&^{1011}_{1010}&^{0121}_{0120}&^{0000}_{0011}&^{0112}_{0111}&^{0110}_{0001}&&^{23}_{24}&^{1112}_{1120}&^{1122}_{1011}&^{1101}_{1222}&^{1233}_{1121}&^{0112}_{0100}&^{1110}_{1123}&^{1222}_{1120}\\
^{4}_{5}&^{0012}_{0100}&^{0012}_{0110}&^{1012}_{1210}&^{0123}_{0000}&^{0001}_{1120}&^{0110}_{1101}&^{0122}_{1100}&&^{25}_{26}&^{1121}_{1122}&^{1010}_{1112}&^{1221}_{1220}&^{1120}_{1223}&^{0121}_{0111}&^{1122}_{1121}&^{0101}_{1122}\\
^{6}_{7}&^{0101}_{0110}&^{0112}_{0100}&^{1212}_{1200}&^{0001}_{0111}&^{1123}_{1110}&^{1100}_{0012}&^{1121}_{0010}&&^{27}_{28}&^{1123}_{1200}&^{1012}_{1220}&^{1233}_{0110}&^{1123}_{1000}&^{0101}_{1230}&^{1120}_{1201}&^{1223}_{1200}\\
^{8}_{9}&^{0111}_{0112}&^{0111}_{0101}&^{1211}_{1201}&^{0110}_{0122}&^{1122}_{1112}&^{0011}_{0010}&^{1111}_{0112}&&^{29}_{30}&^{1201}_{1210}&^{1223}_{1210}&^{0112}_{0100}&^{1001}_{1211}&^{1234}_{1220}&^{1200}_{1012}&^{1231}_{1010}\\
^{10}_{11}&^{0120}_{0121}&^{0122}_{0121}&^{1232}_{1231}&^{0011}_{0010}&^{1100}_{1121}&^{1223}_{1222}&^{1220}_{1101}&&^{31}_{32}&^{1211}_{1212}&^{1222}_{1212}&^{0111}_{0101}&^{1210}_{1232}&^{1233}_{1223}&^{1011}_{1010}&^{1221}_{1212}\\
^{12}_{13}&^{0122}_{0123}&^{0120}_{0123}&^{1230}_{1234}&^{0112}_{0012}&^{1111}_{1101}&^{1221}_{1220}&^{0012}_{1233}&&^{33}_{34}&^{1220}_{1221}&^{1200}_{1221}&^{1122}_{1121}&^{1111}_{1110}&^{1210}_{1232}&^{0123}_{0122}&^{0120}_{1211}\\
^{14}_{15}&^{1000}_{1001}&^{1100}_{1101}&^{0000}_{0001}&^{1200}_{1201}&^{1010}_{1012}&^{1001}_{1000}&^{1000}_{1011}&&^{35}_{36}&^{1222}_{1223}&^{1211}_{1201}&^{1120}_{1223}&^{1222}_{1122}&^{1222}_{1212}&^{0121}_{0120}&^{1112}_{1123}\\
^{16}_{17}&^{1010}_{1011}&^{1121}_{1120}&^{0011}_{0010}&^{1231}_{1230}&^{1000}_{1011}&^{1212}_{1211}&^{1210}_{1001}&&^{37}_{38}&^{1230}_{1231}&^{1233}_{1232}&^{0122}_{0121}&^{1011}_{1010}&^{1200}_{1231}&^{1234}_{1233}&^{1230}_{1201}\\
^{18}_{19}&^{1012}_{1100}&^{1123}_{1000}&^{0012}_{1110}&^{1234}_{1100}&^{1001}_{0120}&^{1210}_{0101}&^{1232}_{0100}&&^{39}_{40}&^{1232}_{1233}&^{1231}_{1230}&^{0120}_{1123}&^{1212}_{1112}&^{1221}_{1211}&^{1232}_{1231}&^{1012}_{0123}\\
^{20}&^{1101}&^{1001}&^{1112}&^{1101}&^{0123}&^{0100}&^{0121}&&^{41}&^{1234}&^{1234}&^{0123}&^{1012}&^{1201}&^{1230}&^{1234}\\
^{-}&^{--}&^{--}_{4***}&^{--}_{****}&^{--}_{4***}&^{--}_{*3**}&^{--}_{**2*}&^{--}_{***1}
&&
^{-}&^{--}&^{--}_{4***}&^{--}_{****}&^{--}_{4***}&^{--}_{*3**}&^{--}_{**2*}&^{--}_{***1}
\end{array}$$\end{center}

The sequences in Theorem~\ref{nt} (B)-(C) start as follows, with intervals ended in ``;'':

$$\begin{array}{rrrrrrrrrrrrrrrrrr}
^{\{0\}\cup\mathbb{Z}^+=}&^{0,}&^{1;}&^{2,}&^{3,}&^{4;}&^{5,}&^{6,}&^{7,}&^{8,}&^{9,}&^{10,}&^{11,}&^{12,}&^{13;}&^{14}&^{15,}&^{16,\ldots}\\\hline
^{(B)=}_{(C)=}&^{0,}_{1,}&^{1;}_{0;}&^{3,}_{0,}&^{2,}_{3,}&^{4;}_{1;}&^{7,}_{0,}&^{9,}_{1,}&^{5,}_{8,}&^{8,}_{7,}&^{\hspace{1mm}6,}_{12,}&^{12,}_{\hspace{1mm}3,}&^{11,}_{\hspace{1mm}2,}&^{10,}_{\hspace{1mm}9,}&^{13;}_{\hspace{1mm}4;}&^{19,}_{\hspace{1mm}0,}&^{20,}_{\hspace{1mm}1,}&^{25,\ldots}_{\hspace{1mm}3,\ldots}\\
\end{array}$$

\begin{remark}\label{pro1} With the notation of Section~\ref{rootk} and Theorem~\ref{pro}, for each of the involutions $\alpha^i$ ($0<i<k$), it holds that $\alpha^i\Theta=\Theta\alpha^{k-i}$. This implies that 

{\bf(A)} every $0$-colored edge represents an adjacency via $\Phi'=F\alpha_kF^{-1}$ and 

{\bf(B)} every $1$-colored edge represents an adjacency via $\Psi'=F\Psi F^{-1}$. 

\noindent In addition,
the reflection symmetry of $\Phi'$ yields the sequence $S_0$ cited in Theorem~\ref{nt}(B). A similar observation yields from $\Psi'$ the sequence $S_1$ cited in Theorem~\ref{nt}(C).
\end{remark}

Given a $k$-germ $\alpha=a_{k-1}\cdots a_1$, we want to express $\alpha^k,\alpha^{k-1},\ldots,\alpha^0$ as functions of $\alpha$.
Given a substring $\alpha'=a_{k-j}\cdots a_{k-i}$ of $\alpha$ ($0<j\le i<k$), let: 
{\bf(a)} the {\it reverse string} off $\alpha'$ be $\psi(\alpha')=a_{k-i}\cdots a_{k-j}$; 
{\bf(b)} the {\it ascent} of $\alpha'$ be 
{\bf(i)} its maximal initial ascending substring, if $a_{k-j}=0$, and 
{\bf(ii)} its maximal initial non-descending substring with at most 2 equal nonzero
terms, if $a_{k-j}>0$.
Then, the following remarks allow to express the $k$-germs $\alpha^p=\beta=b_{k-1}\cdots b_1$ via the colors $p=0,1,\ldots,k$, independently of $F^{-1}$ and $F$.

\begin{remark}\label{p10} Assume $p=k$. If $a_{k-1}=1$, take $0|\alpha$ instead of $\alpha=a_{k-1}\cdots a_1$, with $k-1$ instead of $k$, removing afterwards from the resulting $\beta$ the added leftmost 0. Now, let $\alpha_1=a_{k-1}\cdots a_{k-i_1}$ be the ascent of $\alpha$. Let $B_1=i_1-1$, where $i_1=||\alpha_1||$ is the length of $\alpha_1$.
It can be seen that $\beta$ has ascent $\beta_1=b_{k-1}\cdots b_{k-i_1}$
with $\alpha_1+\psi(\beta_1)=B_1\cdots B_1$. If $\alpha\ne\alpha_1$, let
$\alpha_2$ be the ascent of $\alpha\setminus\alpha_1$. Then
there is a $||\alpha_2||$-germ $\beta_2$ with $\alpha_2+\psi(\beta_2)=B_2\cdots B_2$ and  $B_2=||\alpha_1||+||\alpha_2||-2$. Inductively when feasible for $j>2$, let $\alpha_j$ be the ascent of $\alpha\setminus(\alpha_1|\alpha_2|\cdots|\alpha_{j-1})$. Then there is a $||\alpha_j||$-germ $\beta_j$ with $\alpha_j+\psi(\beta_j)=B_j\cdots B_j$ and $B_j=||\alpha_{j-1}||+||\alpha_j||-2$. This way, $\beta=\beta_1|\beta_2|\cdots|\beta_j|\cdots$.
\end{remark}

\begin{remark}\label{p11} Assume $k>p>0$. By Theorem~\ref{nt} (A), if $p<k-1$, then $b_{p+1}=a_{p+1}$; in this case, let $\alpha'=\alpha\setminus\{a_{k-1}\cdots a_q\}$ with $q=p+1$. If $p=k-1$, let $q=k$ and let $\alpha'=\alpha$.
In both cases (either $p<k-1$ or $p=k-1$) let $\alpha'_1=a_{q-1}\cdots a_{k-i_1}$ be the ascent of $\alpha'$. It can be seen that $\beta'=\beta\setminus\{b_{k-1}\cdots b_q\}$ has ascent $\beta'_1=b_{k-1}\cdots b_{k-i_1}$ where $\alpha'_1+\psi(\beta'_1)=B'_1\cdots B'_1$ with $B'_1=i_1+a_q$.
If $\alpha'\ne\alpha'_1$ then let $\alpha'_2$ be the ascent of $\alpha'\setminus\alpha'_1$. Then
there is a $||\alpha'_2||$-germ $\beta'_2$ where $\alpha'_2+\psi(\beta'_2)=B'_2\cdots B'_2$ with  $B'_2=||\alpha'_1||+||\alpha'_2||-2$. Inductively when feasible for $j>2$, let $\alpha_j$ be the ascent of $\alpha'\setminus(\alpha'_1|\alpha'_2|\cdots|\alpha'_{j-1})$. Then there is a $||\alpha'_j||$-germ $\beta'_j$ where $\alpha'_j+\psi(\beta'_j)=B'_j\cdots B'_j$ with $B'_j=||\alpha'_{j-1}||+||\alpha'_j||-2$. This way, $\beta'=\beta'_1|\beta'_2|\cdots|\beta'_j|\cdots$.

We process the left-hand side from position $q$. If $p>1$, we set $a_{a_q+2}\cdots a_q +\psi(b_{b_q+2}\cdots b_q)$ to equal a constant string $B\cdots B$, where $a_{a_q+2}\cdots a_q$ is an ascent and $a_{a_q+2}=b_{b_q+2}$.
Expressing all those numbers $a_i,b_i$ as $a_i^0,b_i^0$, respectively, in order to keep an inductive approach, let $a^1_q=a_{a_q+2}$. While feasible, let $a^1_{q+1}=a_{a_q+1}$, $a^1_{q+2}=a_{a_q}$ and so on. In this case, let $b^1_q=b_{b_q+2}$, $b^1_{q+1}=b_{b_q+1}$, $b^1_{q+2}=b_{b_q}$ and so on. Now, $a^1_{a^1_q+2}\cdots a^1_q+\psi(b^1_{b^1_q+2}\cdots b^1_q)$ equals a constant string, where $a^1_{a^1_q+2}\cdots a^1_q$ is an ascent and $a^1_{a^1_q+2}=b^1_{b^1_q+2}$.  The continuation of this procedure produces a subsequent string $a^2_q$ and so on, until what remains to reach the leftmost entry of $\alpha$ is smaller than the needed space for the procedure itself to continue, in which case, a remaining initial ascent is shared by both $\alpha$ and $\beta$. This allows to form the left-hand side of $\alpha^p=\beta$ by concatenation.
\end{remark}


\begin{thebibliography}{99}

\bibitem{Arndt} J. Arndt, Matters Computational: Ideas, Algorithms, Source Code, Springer, 2011.

\bibitem{gmn} P. Gregor, T. M\"utze and J. Nummenpalo, {\it A short proof of the middle levels theorem}, Discrete Analysis, 2018:8, 12pp.

\bibitem{KT} H.\ A.\ Kierstead and W.\ T.\ Trotter, {\it Explicit matchings in the middle levels of the Boolean lattice}, Order, {\bf 5} (1988), 163--171.

\bibitem{M} T. M\"{u}tze, {\it Proof of the middle levels conjecture}, Proc. LMS, {\bf112} (2016) 677--713.

\bibitem{Savage} I. Shields and C. Savage, {\it A Hamilton path heuristic with applications to the middle 2 levels problem}, Congr. Num., {\bf 140} (1999), 161-–178.

\bibitem{oeis} N.\ J.\ A.\ Sloane, The On-Line Encyclopedia of Integer Sequences, \url{http://oeis.org/}.

\bibitem{Stanley} R. Stanley, Enumerative Combinatorics, Volume 2, (Cambridge Studies in Advanced Mathematics Book 62), Cambridge University Press, 1999.
\end{thebibliography}
\end{document}